\documentclass{article}

\usepackage[a4paper, total={7in, 8in}]{geometry}
\usepackage[utf8]{inputenc}   % LaTeX, comprends les accents !
\usepackage[T1]{fontenc}      % Police contenant les caractÃ¨res franÃ§ais
\usepackage{multicol}
\usepackage{amsmath,amsthm} 
\usepackage{amssymb,mathrsfs} 
\usepackage{amssymb}
\usepackage{amsfonts}
\usepackage[normalem]{ulem}
\usepackage{nicefrac}
\usepackage{graphicx}
\usepackage{enumerate}
\usepackage{graphicx} %insertion d'images 
\usepackage{lmodern} %% improve pdf rendering
\usepackage{mathtools}
\usepackage{dsfont}
\usepackage{physics}
\usepackage{framed}
\usepackage{color}
\usepackage{setspace}

\usepackage[usenames,dvipsnames]{pstricks}
\usepackage{epsfig}
\usepackage{pst-grad} % For gradients
\usepackage{pst-plot} % For axes
\usepackage[space]{grffile} % For spaces in paths
\usepackage{etoolbox} % For spaces in paths
\makeatletter % For spaces in paths
\patchcmd\Gread@eps{\@inputcheck#1 }{\@inputcheck"#1"\relax}{}{}
\makeatother

\newcommand\N{\mathbb{N}}

\newcommand\R{\mathbb{R}}
\newcommand\Z{\mathbb{Z}}

\newcommand\Qhdt{Q_{h,\Delta t}}
\newcommand\Qdt{Q_{\Delta t}}
\newcommand\Lc{\mathcal{L}}

\newcommand\X{\mathcal{X}}
\newcommand\Kb{\mathbb{K}}

\newcommand\B{\mathcal{B}}
\newcommand\PX{\mathcal{P}(\X)}

\newcommand\e{\mathrm{e}}

\newcommand\Linfty{B^{\infty}}
\newcommand\id{\mathrm{Id}}
\newcommand\E{\mathbb{E}}

\newcommand\Dt{\Delta t}
\newcommand\intX{\int_{\X}}
\newcommand\hdt{h_{\Dt}}
\newcommand\lambdadt{\lambda_{\Dt}}

\newcommand\ind{\mathds{1}}
\newcommand\nb{{\bar{n}}}

\newcommand{\vertiii}[1]{{\left\vert\kern-0.25ex\left\vert\kern-0.25ex\left\vert #1 
    \right\vert\kern-0.25ex\right\vert\kern-0.25ex\right\vert}}

\DeclarePairedDelimiter\ceil{\lceil}{\rceil}

\renewcommand{\leq}{\leqslant}

\renewcommand{\geq}{\geqslant}

\newtheorem{theorem}{Theorem}
\newtheorem{lemma}{Lemma}
\newtheorem{assumption}{Assumption}
\newtheorem{prop}{Proposition}

\newtheorem{remark}{Remark}
\newtheorem{definition}{Definition}

\iffalse    %--- corrige apparent ---
   \excludecomment{corrige}
\else
  
\fi

\graphicspath{{fig/}}

\begin{document}

\title{More on the long time stability of Feynman--Kac semigroups}
\author{Grégoire Ferré$^{1}$, Mathias Rousset$^2$ and Gabriel Stoltz$^1$\\
\small (1) Université Paris-Est, CERMICS (ENPC), Inria, F-77455 Marne-la-Vallée, France \\
\small (2) INRIA Rennes -- Bretagne Atlantique  \& IRMAR Université Rennes 1, France
}

\date{\today}

\maketitle

%%%%%%%%%%%%%%%%%%%%%%%%%%%%%%%%%%%%%%%%%%%%%%%%%%%%%%%%%%%%%%%%%%%%%%%%%%%%%%%%%%%%%%%%%%%
\abstract{
  Feynman--Kac semigroups appear in various areas of mathematics: non-linear filtering, large deviations theory, spectral analysis of Schrödinger operators among others. Their long time behavior provides important information, for example in terms of ground state energy of Schrödinger operators, or scaled cumulant generating function in large deviations theory. In this paper, we propose a simple and natural extension of the stability analysis of Markov chains for these non-linear evolutions. As other classical ergodicity results, it relies on two assumptions: a Lyapunov condition that induces some compactness, and a minorization condition ensuring some mixing.  We show that these conditions are satisfied in a variety of situations, including stochastic differential equations. Illustrative examples are provided, where the stability of the non-linear semigroup arises either from the underlying dynamics or from the Feynman--Kac weight function.  We also use our technique to provide uniform in the time step convergence estimates for discretizations of stochastic differential equations.
}

\vspace{1cm}
\textbf{Key words} : Feynman--Kac dynamics; ergodicity; spectral analysis; large deviations.
%%%%%%%%%%%%%%%%%%%%%%%%%%%%%%%%%%%%%%%%%%%%%%%%%%%%%%%%%%%%%%%%%%%%%%%%%%%%%%%%%%%%%%%%%%%
\section{Introduction}
\label{sec:intro}

Feynman--Kac semigroups have a long history in physics and mathematics. One of their traditional
applications as a probabilistic representation of Schrödinger semigroups~\cite{karatzas2012brownian}
is the computation of ground state energies through Diffusion Monte Carlo
algorithms~\cite{grimm1971monte,anderson1975random,ceperley1980ground,foulkes2001quantum}.
It has also become a significant tool in non-linear filtering and genealogical
models~\cite{del2000branching,del2003particle,del2004feynman}, as well as in large deviations
theory~\cite{donsker1975variational,kontoyiannis2005large,giardina2006direct,wu2001large}.
In all these contexts, the dynamics is evolved and its paths are weighted depending on
some cost function. This function is typically a potential energy, a likelihood, or a function whose
fluctuations are of interest.

As for Markov chains, the long time behavior of such dynamics is important. However, the long-time
analysis is made difficult by the non-linear character of the evolution, so the methods used
for the stability of Markov chains~\cite{meyn2012markov,hairer2011yet} cannot be straightforwardly
adapted in this context. A series of papers~\cite{del2001stability,del2002stability,del2004feynman}
rely on the powerful Dobrushin ergodic
coefficient~\cite{dobrushin1956centralI,dobrushin1956centralII}. However, although this tool enables to
deal with the nonlinearity and to consider time-inhomogeneous processes, the conditions imposed
on the dynamics are not realistic for unbounded domains.

The purpose of this paper is to propose a new scheme of proof for the ergodicity of Feynman--Kac
dynamics, suitable for cases where the state space is unbounded.  It is based on 
the principal eigenvalue problem associated to a weighted evolution operator.
It then relies on studying a
$h$-transformed version of the dynamics~\cite{doob1957conditional}, where~$h$ is the eigenvector
associated to the eigenproblem. This turns the non-linear dynamics into a linear Markov
evolution, which can then be studied with standard techniques~\cite{meyn2012markov,hairer2011yet}.
However, the spectral properties of the generator fall out of the typical regime of self-adjoint
operators, since the dynamics is in general non-reversible. A striking fact of
our results is that, under Lyapunov and minorization conditions similar to
those of~\cite{hairer2011yet} stated for non-probabilistic kernels, we perform a non self-adjoint
spectral analysis that recasts the Feynman--Kac problem into the Markov chain
framework studied in~\cite{hairer2011yet}.

The works of Kontoyannis and Meyn~\cite{kontoyiannis2003spectral,kontoyiannis2005large} provide
elements of answer concerning the spectral properties of the evolution operator, and
rely on a nonlinear Lyapunov condition and a regularity in terms
of hitting times. If the latter Lyapunov condition is natural in terms of optimal stochastic
control~\cite{fleming1977exit}, we propose instead proofs based on linear conditions.
Our generalized linear Lyapunov condition is inspired by~\cite{bellet2006ergodic}, and comes
together with a minorization condition and a local strong Feller assumption. We will see that these
conditions apply to a variety of situations,
with natural interpretations. From a broader perspective, it appears as a natural extension
of previous works on the stability of Markov chains~\cite{hairer2011yet} for
evolution kernels that do not conserve probability. To that extent, our work resonates with recent
works on Quasi-Stationary Distributions
(QSD)~\cite{gosselin2001asymp,champagnat2017lyapunov,champagnat2017general,bansaye2017ergodic}. However, our
scope and assumptions being different, we leave the comparison for future studies.
Let us also mention that our framework applies for both discrete and continuous time processes.
This is interesting since one motivation for this work is to understand the behavior of time
discretizations of continuous Feynman--Kac dynamics, as in~\cite{ferre2017error}.

Let us outline our main results in an informal way. The quantities we are interested in
typically correspond to Markov chains~$(x_k)_{k\geq 0}$ over a state space~$\X$, whose
trajectories are weighted by a function  $f:\X\to \R$. This corresponds to semigroups of the form
\begin{equation}
  \label{eq:feynmankac}
\Phi_k(\mu)(\varphi)=\frac{\E \left[ \varphi(x_k)\, \e^{ \sum_{i=0}^{k-1} f(x_i)} \ \Big|\
    x_0\sim \mu \right]}
{\E \left[ \e^{\sum_{i=0}^{k-1} f(x_i)} \ \Big|\
    x_0\sim \mu \right]},
\end{equation}
where~$\mu$ is an initial probability distribution, and~$\varphi$ is a test function. 
We show that, for more general semigroups and under some assumptions on~$(x_k)_{k\geq 0}$
and~$f$, there exists a measure~$\mu_f^\star$ such that for any initial measure~$\mu$ and
any~$\varphi$ belonging to a particular class of unbounded test functions,
\begin{equation}
  \label{eq:ergoexample}
\Phi_k(\mu)(\varphi) \xrightarrow[k\to+\infty]{} \mu_f^\star(\varphi),
\end{equation}
at an exponential rate. As a corollary of this result, we show that the principal
eigenvalue~$\Lambda$ of the generator of the dynamics~$( \Phi_k)_{k\geq 1}$ can be
obtained as the following limit, for any initial measure~$\mu$ and suitable functions~$f$:
\[
\log(\Lambda) = \underset{k\to+\infty}{\lim}\ \frac{1}{k}\log \E \left[ \e^{\sum_{i=0}^{k-1}
    f(x_i)} \ \Big|\ x_0\sim \mu \right],
\]
  a quantity sometimes called scaled cumulant generating function in large deviations
  theory~\cite{dembo2010large,kontoyiannis2005large}. 
  Another natural situation corresponds to continuous semigroups of the form
  \begin{equation}
    \label{eq:excontinuous}
\Theta_t(\mu)(\varphi)  =
\frac{\E\left[ \varphi(X_t)\, \e^{\int_0^t f(X_s)\, ds} \ \Big|\ X_0\sim \mu \right]}
     {\E\left[ \e^{ \int_0^t f(X_s)\, ds} \ \Big|\ X_0\sim \mu \right]},
     \end{equation}
where~$(X_t)_{t\geq 0}$ is typically a diffusion process.
Results similar to the ones obtained in the discrete time setting are then derived
for this continuous dynamics. We will see that
ergodic properties such as~\eqref{eq:ergoexample} are proved under natural extensions of
Lyapunov and minorization conditions, which should be reminiscent of the corresponding theory for
Markov chains~\cite{hairer2011yet,bellet2006ergodic}, with additional regularity
conditions.

The paper is organized as follows. In Section~\ref{sec:results}, we present
our main results on the stability of Feynman--Kac semigroups. Section~\ref{sec:discrete} is
devoted to discrete time results, while Section~\ref{sec:continuous} is concerned with
the continuous time case. Section~\ref{sec:applications} presents a number of natural
applications of the method. In particular, Section~\ref{sec:discretization} provides
uniform in the time step convergence estimates. Section~\ref{sec:discussion} discusses
some links with related works and possible further directions.

%%%%%%%%%%%%%%%%%%%%%%%%%%%%%%%%%%%%%%%%%%%%%%%%%%%%%%%%%%%%%%%%%%%%%%%%%%%%%%%%%%%%%%%%%%%%%%
%%%%%%%%%%%%%%%%%%%%%%%%%%%%%%%%%%%%%%%%%%%%%%%%%%%%%%%%%%%%%%%%%%%%%%%%%%%%%%%%%%%%%%%%%%%

\section{Results}
\label{sec:results}

\subsection{Framework}
In this section, we present our main convergence results for generalizations of the
dynamics~\eqref{eq:feynmankac}. The state space~$\X$ is assumed to be a Polish space, and
for a measurable set $A\subset\X$, we denote by~$A^c$ its complement, and~$\ind_A$ its indicator
function. For a Banach space~$E$, we denote by~$\mathcal{B}(E)$ the space of bounded linear
operators over~$E$, with associated norm
$\| T \|_{\mathcal{B}(E)}=\sup\, \{ \| Tu\|_E,\, \|u\|_E\leq 1 \}$.  The Banach space
of continuous functions is called~$C^0(\X)$, and the Banach space of measurable
functions~$\varphi$ such that
\[
\| \varphi \|_{\Linfty} := \underset{x\in\X}{\sup}\ | \varphi(x) | < +\infty
\]
is referred to as~$\Linfty(\X)$. Given a measure~$\mu$ over~$\X$ with finite mass,
we use the notation $\mu(\varphi)= \intX \varphi(x) \mu(dx)$ for $\varphi\in \Linfty(\X)$.
The spaces of positive measures and probability measures over~$\X$ are denoted respectively
by~$\mathcal{M}(\X)$ and~$\PX$. When we consider Markov chains~$(x_k)_{k\in\N}$ over~$\X$, we
write~$\E_{\mu}$ for the expectation over all the realizations of the
Markov chain with initial condition distributed according to the probability measure~$\mu$.
Appendix~\ref{sec:tools} is devoted to
reminders on the ergodicity of Markov chains extracted from~\cite{hairer2011yet},
while Appendix~\ref{sec:krein} recalls some useful definitions and theorems used in the proofs
of the results of this section.

We consider general kernel operators~$Q^f$ over~$\X$, \textit{i.e.} such that for
any $x\in\X$,~$Q^f(x,\cdot)$ is a
positive measure with finite mass (\textit{i.e.}~$Q^f\ind (x) < +\infty$), and for
any measurable set $A\subset\X$,~$Q^f(\cdot,A)$
is a measurable function. Such a kernel is referred to as Markov (also probabilistic or
conserving) when $Q^f \ind = \ind$. The notation~$Q^f$ instead of~$Q$ emphasizes that in general
the function $Q^f\ind\neq \ind$ depends on a measurable function $f:\X\to\R$.
For $\varphi\in\Linfty(\X)$, we denote by $Q^f\varphi= \int_{\X} \varphi(y)Q^f(\cdot,dy)$
the action of~$Q^f$ on test functions, and by $\mu Q^f=\int_{\X}\mu(dx)Q^f(x,\cdot)$ its action
on finite measures~$\mu$. We call Feynman--Kac semigroups
the dynamics~$( \Phi_k)_{k\geq 1}$ defined as follows:
\begin{equation}
  \label{eq:dynamics}
\forall\, k\geq 1, \quad   \forall\, \mu\in\PX,\quad \forall\, \varphi\in \Linfty(\X), \quad
  \Phi_k(\mu)(\varphi)= \frac{\mu\big( (Q^f)^k \varphi\big)}{\mu\big( (Q^f)^k
    \ind \big)}.
\end{equation}
Note that $\Phi_k = \Phi \circ \hdots \circ \Phi$, where~$\Phi$ is the one step evolution
operator $\Phi:\PX\to\PX$:
\begin{equation}
  \label{eq:onestep}
  \forall\, \mu\in\PX,\quad \forall\,
  \varphi\in \Linfty(\X), \quad
  \Phi(\mu)(\varphi) = \frac{\mu\big( Q^f \varphi \big) } { \mu\big( Q^f \ind \big)},
\end{equation}
which is well-defined as soon as~$\mu(Q^f\ind) >0$ for any $\mu\in\PX$. Lemma~\ref{lem:preliminary}
below proves that~\eqref{eq:onestep} is indeed well-defined under the assumptions presented in
Section~\ref{sec:discrete}.

Although~$Q^f$ is not probabilistic, the normalizing factor in~\eqref{eq:onestep} ensures
that~$\Phi$ evolves a positive measure of finite mass into a probability measure. An important 
motivation for studying the general dynamics~\eqref{eq:onestep} is that~\eqref{eq:feynmankac} can
be written in the form~\eqref{eq:dynamics} with $Q^f=\e^{f}Q$, where~$Q$ is the transition operator
of the Markov chain~$(x_k)_{k\in\N}$. In this typical setting, $Q^f\ind = \e^{f}$. Even
when~$Q^f$ is not defined in this way (see for instance the continuous time
situation~\eqref{eq:contQP} considered in Section~\ref{sec:continuous}), we keep the
notation to emphasize that~$Q^f$ typically corresponds
to a Markov dynamics whose trajectories are weighted by a function~$f$.

\subsection{Results in discrete time}
\label{sec:discrete}
We now introduce the assumptions ensuring the well-posedness and ergodicity of the
semigroup~\eqref{eq:dynamics}, which should be reminiscent of the ones used
in~\cite{hairer2011yet,bellet2006ergodic}
for showing the ergodicity of Markov chains. The first step of the proof is the existence
of a principal eigenvector~$h$ for~$Q^f$, as shown in Lemma~\ref{lem:specQV}. This eigenvector
is used in Lemma~\ref{lem:Qh} to study a $h$-transformed version of~$Q^f$, which leads to our
main result, Theorem~\ref{theo:general}.  Note that, in practice, we have in mind the situation
 $\X=\R^d$ for $d\in\N^*$, but discrete spaces like $\X=\Z^d$ can also be considered, in which case
the framework may be simplified.

The first assumption is that a generalized Lyapunov condition holds. We will see in
Section~\ref{sec:applications} that it is satisfied for a large class of processes.
In all this section, we consider an increasing sequence of compact
sets~$(K_n)_{n\geq 1}$ such that, for any compact $K\subset\X$, there exists $m\geq 1$
for which $K \subset K_m$.

\begin{assumption}[Lyapunov condition]
  \label{as:lyapunovgeneral}
  
There exist a function $W:\X\to [1, +\infty)$ bounded on compact sets, and positive
  sequences~$(\gamma_n)_{n\geq 1}$, $(b_n)_{n\geq 1}$ with $\gamma_n\to 0$ as $n\to +\infty$
  such that, for all $n\geq 1$,
\begin{equation}
\label{eq:lyapunovgeneral}
  Q^f W \leq \gamma_n W + b_n\ind_{K_n}.
\end{equation} 
\end{assumption}
Let us mention that, in many situations, the function~$W$ has compact level sets, so that
a natural choice of compact sets is~$K_n=\{ x\in\X\,|\, W(x) \leq n\}$.
When a Lyapunov function~$W$ exists, it is natural~\cite{hairer2011yet} to consider the following
functional space
\begin{equation}
  \label{eq:LW}
\Linfty_W(\X) = \Big\{ \varphi \mbox{ measurable}, \ \ \left\| \frac{\varphi}{W} \right\|_{\Linfty}
<  + \infty \Big\}.
\end{equation}
In particular, Assumption~\ref{as:lyapunovgeneral} implies that~$Q^f$ is a bounded operator
on~$\Linfty_W(\X)$, since one can show that
\[
\forall\, n\geq 1, \quad
\|Q^f\|_{\mathcal{B}(\Linfty_W)}\leq \gamma_n +b_n.
\]
We next assume that the following minorization condition holds.
%%
% \forall\, x \in K_n, \quad
\begin{assumption}[Minorization and irreducibility]
  \label{as:minogeneral}
  For any $n\geq 1$, there exist $\eta_n \in \PX$ and $\alpha_n >0$ such that
  \begin{equation}
    \label{eq:minogeneral}
    \inf_{x\in K_n}\,     Q^f(x, \cdot) \geq \alpha_n \eta_n(\, \cdot\,).
  \end{equation}
  In addition, for any $n_0\geq 1$ and any $\varphi\in\Linfty_W(\X)$ with $\varphi\geq 0$,
  \begin{equation}
    \label{eq:irreducibility}
    \eta_{n}(\varphi) = 0, \, \forall\, n\geq n_0 \ \Longrightarrow \
    \big(Q^f\varphi\big)(x) = 0,\, \forall\,x\in\X.
  \end{equation}
\end{assumption}
Note that~\eqref{eq:irreducibility} expresses some form of irreducibility with
respect to the minorizing measures. It can be reformulated in the following way: for
any~$n_0\geq 1$ and any~$x\in\X$, $Q^f(x,\cdot)$ is absolutely continuous with
respect to the measure
\[
\sum_{n\geq n_0} 2^{-n}\eta_n.
\]
The typical situation for $\X=\R^d$ is
to choose $\eta_n(dx) = \ind_{K_n}(x)dx/|K_n|$, where~$|K_n|$ denotes the Lebesgue
measure of~$K_n$. We also mention that, although we will consider the previous minorization
measures~$\eta_n$ in our examples in Section~\ref{sec:applications}, the first part of
Assumption~\ref{as:minogeneral} can be obtained using irreducibility
together with a strong Feller property, see~\cite{hairer2001exponential}, or through
the Stroock--Varadhan support theorem~\cite{stroock1972support} with some regularity
property, see the discussion in~\cite{bellet2006ergodic}.
In our context, we also need some local regularity for the operator~$Q^f$.
\begin{assumption}[Local regularity]
\label{as:regularity}
The operator~$Q^f$ is strong Feller on the compact sets~$K_n$, \textit{i.e.} for any $n\geq 1$
and any measurable function~$\varphi$ bounded on~$K_n$, $Q^f(\varphi\ind_{K_n})$ is continuous
over~$K_n$.
\end{assumption}

From these assumptions we first state the following preliminary lemma, whose proof can be found
in Appendix~\ref{sec:preliminary}.
\begin{lemma}
  \label{lem:preliminary}
  Let~$Q^f$ satisfy Assumptions~\ref{as:lyapunovgeneral} and~\ref{as:minogeneral}.
  Then, for any $\mu\in \PX$ with $\mu(W)<+\infty$, one has
  \begin{equation}
    \label{eq:positivity}
  0< \mu(Q^f \ind ) < + \infty.
  \end{equation}
  Moreover, for any $n\geq 1$, it holds $1 \leq \eta_n(W) < +\infty$, and
  there exist infinitely many indices $\nb \geq 1$ such that
  \begin{equation}
    \label{eq:nbar}
  \eta_\nb(K_\nb) >0.
  \end{equation}
\end{lemma}
The lower bound in~\eqref{eq:positivity} implies in particular that the
dynamics~\eqref{eq:dynamics} is well-defined. The inequality~\eqref{eq:nbar}
means that, for infinitely many minorization conditions, some mass of the minorizing
measure remains in the associated compact set. It is used in the proof of Lemma~\ref{lem:specQV} to
show that~$Q^f$ has a positive spectral radius. Since~\eqref{eq:nbar} is satisfied for
infinitely many indices, we could consider that it holds for any~$n\geq 0$,
upon extracting a subsequence and, in the situations considered in
Section~\ref{sec:applications}, we can actually check that $\eta_n(K_n)>0$ for all
$n\geq 1$.

We are now in position to state some spectral properties of the operator~$Q^f$,
which are a key ingredient for our analysis.
Let us recall that the spectral radius of~$Q^f$ on~$\Linfty_W(\X)$, denoted by
$\Lambda:=\Lambda(Q^f)$, is given by the Gelfand formula~\cite{nussbaum1998eigen}:
\begin{equation}
  \label{eq:gelfand}
  \Lambda=\underset{k\to +\infty} {\lim} \,
  \big\| (Q^f)^k \big\|_{\mathcal{B}(\Linfty_W)}^{\frac{1}{k}},
\end{equation}
and that the essential spectral radius of~$Q^f$, denoted by~$\theta(Q^f)$, reads
(see Appendix~\ref{sec:krein}):
  \[
  \theta(Q^f) = \underset{k\to+\infty}{\lim}\, \Big(\inf \big\{ \big\| (Q^f)^k - T \big\|_{\mathcal{B}(\Linfty_W)},\ T
  \ \mathrm{ compact} \big\}\Big)^{\frac{1}{k}}.
  \]

\begin{lemma}
\label{lem:specQV}
Under Assumptions~\ref{as:lyapunovgeneral},~\ref{as:minogeneral} and~\ref{as:regularity},
the operator~$Q^f$ considered on~$\Linfty_W(\X)$ has a zero essential spectral radius,
admits its spectral radius $\Lambda >0$ as a largest eigenvalue (in modulus), and has an associated
eigenfunction $h\in\Linfty_W(\X)$, normalized so that $\| h \|_{\Linfty_W} = 1$,
and which satisfies
\begin{equation}
  \label{eq:muestimate}
  \forall\, x\in\X, \quad 0 < h(x) < +\infty.
\end{equation}
In particular, $0<\eta_{n}(h)<+\infty$ for all $n\geq 1$.
\end{lemma}
Note that the eigenspace associated with~$\Lambda$ is a priori not of dimension one.
We prove Lemma~\ref{lem:specQV} in Appendix~\ref{sec:specQV} by using arguments
inspired by~\cite[Theorem 8.9]{bellet2006ergodic} to show that the essential spectral radius
of~$Q^f$ is zero, and then relying on the theory of positive operators~\cite{deimling2010nonlinear}.
Some useful elements of operator theory are reminded in Appendix~\ref{sec:krein} for the reader's
convenience. Our result is close to those obtained in~\cite{kontoyiannis2005large}, and the
control of the essential spectral radius under Lyapunov and topological
conditions has already been studied in~\cite{wu2004essential,guibourg2011quasi}.
However, our proof uses different techniques based on different assumptions.

Once such a principal eigenvector~$h$ is available, the geometric ergodicity of the Feynman--Kac
dynamics~\eqref{eq:dynamics} is derived from the one of a $h$-transformed kernel, as
made clear in the proof of Theorem~\ref{theo:general} below. This
is the purpose of the next lemma whose proof is postponed to Appendix~\ref{app:Qh}.

\begin{lemma}
  \label{lem:Qh}
  Suppose that Assumptions~\ref{as:lyapunovgeneral},~\ref{as:minogeneral} and~\ref{as:regularity}
  hold, and consider an eigenvector~$h$ associated with~$\Lambda$ as given by
  Lemma~\ref{lem:specQV}. Since~$h>0$ we can define the corresponding $h$-transformed
  operator~$Q_h$ as
  \begin{equation}
    \label{eq:Qh}
    Q_h\phi= \Lambda^{-1}h^{-1} Q^f( h\phi).
  \end{equation}
  Then~$Q_h$ is a Markov operator with Lyapunov function $Wh^{-1}:\X\to [1,+\infty)$. Moreover,
  there exist a unique $\mu_h\in\PX$, which satisfies $\mu_h(Wh^{-1})<+\infty$, and constants $c>0$,
  $\bar{\alpha}\in (0,1)$ such that, for any $\phi \in \Linfty_{W h^{-1}}(\X)$ and any $k\geq 1$,
  \begin{equation}
    \label{eq:cvQh}
     \big\| Q_h^k\phi - \mu_h(\phi) \big\|_{\Linfty_{Wh^{-1}}} \leq c \bar{\alpha}^k
     \big\| \phi - \mu_h(\phi) \big\|_{\Linfty_{Wh^{-1}}}.
     \end{equation}
\end{lemma}

Although this is not obvious at first glance, the operator~$Q_h$ is in fact
independent of the choice of~$h$ in Lemma~\ref{lem:specQV}, and so
is the invariant measure~$\mu_h$. Actually, Lemma~\ref{lem:Qh} allows to show
that the eigenspace associated with~$h$ has geometric dimension one, \textit{i.e.}
$\mathrm{Ker}\big( Q^f - \Lambda\, \id \big)=\mathrm{Span}\{h\}$. Indeed, if
$\tilde{h}\in\Linfty_W(\X)$ is another eigenvector associated with~$\Lambda$
(which may not be of constant sign), it holds, since $h(x)>0$ for all~$x\in\X$
by~\eqref{eq:muestimate}:
\[
Q_h\left(\frac{\tilde{h}}{h}\right) = \Lambda^{-1}h^{-1} Q^f \tilde{h} =
\frac{\tilde{h}}{h}\in\Linfty_{Wh^{-1}}(\X).
\]
From~\eqref{eq:cvQh}, we obtain
\[
\frac{\tilde{h}}{h} = \mu_h\left(\frac{\tilde{h}}{h}\right),
\]
hence~$\tilde{h}$ is proportional to~$h$. It may actually be possible to directly obtain
this uniqueness result from stronger Krein--Rutman theorems,
like~\cite[Theorem 19.3]{deimling2010nonlinear}, using the irreducibility
condition~\eqref{eq:irreducibility} in Assumption~\ref{as:minogeneral}.

We are now in position to state our main theorem.
\begin{theorem}
\label{theo:general}
Consider a kernel operator~$Q^f$ satisfying 
Assumptions~\ref{as:lyapunovgeneral},~\ref{as:minogeneral} and~\ref{as:regularity} and
the associated dynamics~\eqref{eq:dynamics} with one step evolution operator $\Phi:\PX\to\PX$.
Then~$\Phi$ admits a unique fixed point~$\mu_f^\star\in\PX$, that is a probability measure such that
\begin{equation}
  \label{eq:fixedpoint}
  \Phi (\mu_f^\star) = \mu_f^\star,
\end{equation}
and this measure satisfies $\mu_f^\star(W)<+\infty$. Moreover,
there exists $\bar{\alpha}\in (0,1)$ such that, for any $\mu\in\PX$ satisfying
$\mu(W)<+\infty$, there is $C_{\mu}>0$ for which
\begin{equation}
  \label{eq:FKconv}
  \forall\, \varphi \in \Linfty_W(\X),\quad \forall\, k\geq 1, \quad
  \big| \Phi_k(\mu)(\varphi) - \mu_f^\star(\varphi) \big| \leq
  C_{\mu} \bar{\alpha}^k \| \varphi \|_{\Linfty_W}.
\end{equation}
\end{theorem}

We call~$\mu_f^\star$ the invariant measure of~$Q^f$, in analogy with Markov chains.
Note that Theorem~\ref{theo:general} also implies the convergence of~$\Phi_k(\mu)$
towards~$\mu_f^\star$ in the weighted total variation distance (a special type of Wasserstein
distance~\cite{villani2003topics,hairer2011yet}) defined, for $\mu,\nu\in\PX$
with $\mu(W)<+\infty$, $\nu(W)<+\infty$, by
\begin{equation}
  \label{eq:defrho}
  \rho_W(\mu,\nu)= \underset{\|\varphi\|_{\Linfty_W}\leq 1}{\sup} \
  \intX \varphi(x)\, (\mu-\nu)(dx).
\end{equation}

\begin{proof}
  The key idea of the proof is to reformulate the dynamics~\eqref{eq:dynamics} using the
  $h$-transformed operator $Q_h= \Lambda^{-1}h^{-1} Q^f h$ of Lemma~\ref{lem:Qh}.
  Using the notation of Lemmas~\ref{lem:specQV} and~\ref{lem:Qh},
  we rewrite~\eqref{eq:dynamics} as
  \[
  \Phi_k(\mu)(\varphi)=\frac{\mu\big( (Q^f)^k\varphi\big)\Lambda^{-k}}
     {\mu\big( (Q^f)^k\ind \big)\Lambda^{-k}}
     = \frac{\mu\Big(h \big(\Lambda^{-1} h^{-1} Q^f h\big)^k (h^{-1}\varphi)\Big) }
     {\mu\Big( h \big(\Lambda^{-1} h^{-1} Q^f h\big)^k h^{-1} \Big) }=\frac{\mu\big(h (Q_h)^k (h^{-1}\varphi)\big)}
        {\mu\big(h (Q_h)^k h^{-1} \big)}.
     \]
     The dynamics~\eqref{eq:dynamics} is therefore reformulated
     as the ratio of long time expectations of the Markov chains induced by~$Q_h$, applied
     to the functions~$h^{-1}\varphi$ and~$h^{-1}$. It is then possible to resort to the convergence
     results given by Lemma~\ref{lem:Qh}.

     We first construct a probability measure~$\mu_f^\star$ for which~\eqref{eq:FKconv} is
     satisfied, namely
     \begin{equation}
       \label{eq:mudef}
       \mu_f^\star(\varphi)=\frac{\mu_h\left(h^{-1}\varphi \right)}{\mu_h\left( h^{-1}\right)},
     \end{equation}
     where~$\mu_h$ is the probability measure introduced in Lemma~\ref{lem:Qh}. Note
     that~$\mu_f^\star$
     is well-defined for $\varphi\in\Linfty_{W}(\X)$. Indeed, for $\varphi\in\Linfty_{W}(\X)$,
     it holds $h^{-1}\varphi\in\Linfty_{Wh^{-1}}(\X)$. Second, we show that
     \begin{equation}
       \label{eq:muhpos}
       \mu_h(h^{-1})>0.
     \end{equation}
     Indeed, since $\| h \|_{\Linfty_W}= 1$, it holds $h^{-1}\geq W^{-1}$, and since~$W$ is
     upper bounded on any compact set,~$W^{-1}$ is lower bounded by a positive constant on
     any compact set. As $\mu_h\in\PX$, we can use Lemma~\ref{lem:tight} in
     Appendix~\ref{sec:krein} to conclude that $\mu_h(h^{-1})>0$.
     Moreover,~$\mu_f^\star$ does not depend on the choice of normalization for~$h$.
     Finally, $\mu_f^\star(W)<+\infty$ since $\mu_h(Wh^{-1})<+\infty$.

     From Lemma~\ref{lem:Qh}, for any $\varphi\in\Linfty_W(\X)$, it holds
     $Q_h^k (h^{-1}\varphi)= \mu_h(h^{-1}\varphi)+a_k$ and
     $Q_h^k(h^{-1})= \mu_h(h^{-1})+b_k$ with $\| a_k \|_{\Linfty_{Wh^{-1}}} \leq c \bar{\alpha}^k
     \| h^{-1}\varphi- \mu_h(h^{-1}\varphi) \|_{\Linfty_{Wh^{-1}}}$ and
     $\| b_k \|_{\Linfty_{Wh^{-1}}} \leq
     c \bar{\alpha}^k \| h^{-1} - \mu_h(h^{-1}) \|_{\Linfty_{Wh^{-1}}}$.
     Since $\varphi\in\Linfty_W(\X)$, we have in particular (using also $\|h\|_{\Linfty_W}=1$),
     \[
     \left\| h^{-1}\varphi- \mu_h(h^{-1}\varphi) \right\|_{\Linfty_{Wh^{-1}}}
     \leq \| h^{-1}\varphi \|_{\Linfty_{Wh^{-1}}}
     + \mu_h(h^{-1}|\varphi|) \| h \|_{\Linfty_W}
     \leq \big(1  + \mu_h(Wh^{-1})\big) \| \varphi \|_{\Linfty_{W}}<+\infty.
     \]
     Since $\mu_h(Wh^{-1})<+\infty$, we can set~$c' =1  + \mu_h(Wh^{-1})$
     so that
     \begin{equation}
       \label{eq:an}
     \| a_k \|_{\Linfty_{Wh^{-1}}} \leq c' \bar{\alpha}^k \| \varphi \|_{\Linfty_{W}}.
     \end{equation}
     A similar estimate holds for the sequence~$(b_k)_{k\geq 1}$ by
     taking~$\varphi \equiv \ind$. This leads to, for any $\varphi\in\Linfty_W(\X)$,
     \[
     \begin{aligned}
       \big| \Phi_k(\mu)(\varphi) - \mu_f^\star(\varphi)\big| & = \left|
       \frac{\mu\left(h (Q_h)^k(h^{-1}\varphi)\right)}
        {\mu\left(h (Q_h)^k h^{-1} \right)} - \mu_f^\star(\varphi)\right|
           = \left| \frac{\mu \left(h( \mu_h(h^{-1}\varphi) + a_k) \right)}
           {\mu \left(h( \mu_h(h^{-1}) + b_k) \right)} - \mu_f^\star(\varphi)\right|
           \\ & = \left| \frac{\mu(h)\mu_h(h^{-1}\varphi) +\mu(h a_k) }
           {\mu(h)\mu_h(h^{-1}) + \mu(h b_k) } - \mu_f^\star(\varphi) \right|
             = \left| \frac{\mu_f^\star(\varphi) +c_{\mu,h}\mu(h a_k) } 
           {1 + c_{\mu,h} \mu(h b_k) }- \mu_f^\star(\varphi) \right|,
     \end{aligned}
     \]
     where we introduced
     \begin{equation}
       \label{eq:cmuh}
     c_{\mu,h}= \frac{1}{\mu(h)\mu_h(h^{-1})}.
     \end{equation}
     It holds $0<c_{\mu,h}<+\infty$ because:
     \begin{itemize}
     \item Lemma~\ref{lem:specQV} shows that for any $\mu\in\PX$ with $\mu(W)<+\infty$,
       it holds $0<\mu(h)<+\infty$;
         \item we know that $\mu_h(h^{-1})<+\infty$ from Lemma~\ref{lem:Qh};
         \item $\mu_h(h^{-1}) >0$ by~\eqref{eq:muhpos}.
     \end{itemize}
     Now, since $|b_k|\leq \| b_k\|_{\Linfty_{Wh^{-1}}} W h^{-1}$
     and~\eqref{eq:an} holds for~$b_k$ with $\varphi\equiv \ind$, we have
     \[
     1 + c_{\mu,h}\mu (h b_k) \geq 1 - c_{\mu,h}\mu ( h | b_k| )
     \geq 1 - c_{\mu,h}\mu ( W )  \| b_k\|_{\Linfty_{Wh^{-1} }}
     \geq 1 - \bar{\alpha}^k c' c_{\mu,h}\mu ( W ).
     \]
     Therefore, the choice
     \[
     k\geq - \frac{ \log \big( 2 c' c_{\mu,h} \mu(W)\big) }
     { \log(\bar{\alpha})}
     \]
     ensures that
     \[
     1 + c_{\mu,h}\mu (h b_k) \geq \frac{1}{2}.
     \]
     As a result, for~$k$ large enough, using
     $|a_k| \leq \| a_k \|_{\Linfty_{Wh^{-1}}} Wh^{-1}$ and
     recalling~\eqref{eq:an},
     \begin{equation}
       \label{eq:intmuW}
       \big| \Phi_k(\mu)(\varphi) - \mu_f^\star(\varphi)\big|
       \leq \frac{  c_{\mu,h} \big( \mu(h |a_k| ) + \mu_f^\star(|\varphi|)\mu(h |b_k|)\big) }
            {1 + c_{\mu,h} \mu(h b_k)}
             \leq C_{\mu} \|\varphi\|_{\Linfty_W} \bar{\alpha}^k,
     \end{equation}
     with
     \begin{equation}
       \label{eq:Cmulast}
       C_{\mu} = 2 c_{\mu,h} c' \mu(W)\big(1 +  \mu_f^\star(W) \big) = \frac{2}{\mu_h(h^{-1})}
       \big( 1 + \mu_h(Wh^{-1})\big) \big( 1 + \mu_f^\star(W)\big) \frac{\mu(W)}{\mu(h)}.
     \end{equation}
     We therefore obtain~\eqref{eq:FKconv} from \eqref{eq:intmuW} with the constant
     defined in~\eqref{eq:Cmulast}. Note that~$C_{\mu}$ depends on the initial
     measure~$\mu$ only through the ratio~$\mu(W)/\mu(h)$.

     Taking the supremum over $\varphi\in\Linfty_W(\X)$ such
     that $\|\varphi\|_{\Linfty_W}\leq 1$,~\eqref{eq:intmuW} rewrites, with~\eqref{eq:defrho}:
     \[
     \rho_W\big( \Phi_k(\mu), \mu_f^\star \big) \leq C_{\mu} \bar{\alpha}^k.
     \]
     Choosing $\mu=\Phi(\mu_f^\star)$ and using the semigroup property we obtain
     \[
     \rho_W\big( \Phi(\Phi_k(\mu_f^\star)), \mu_f^\star \big) \leq C_{\mu_f^\star} \bar{\alpha}^k.
     \]
     Taking the limit $k\to+\infty$ shows that $\Phi(\mu_f^\star)=\mu_f^\star$,
     so~$\mu_f^\star$ is a
     fixed point of~$\Phi$.

     We have shown the existence of an invariant measure of the form~\eqref{eq:mudef}, which
     is a fixed point of~$\Phi$ and integrates~$W$. We now turn to uniqueness, which follows
     by a standard fixed point argument. Assume that we have two probability
     measures~$\mu_1$ and~$\mu_2$ satisfying~\eqref{eq:FKconv} and such that
      $\mu_1(W)<+\infty$,  $\mu_2(W)<+\infty$, which are therefore fixed
     points of~$\Phi$. Then, there exists $\bar{\alpha}\in(0,1)$ such that, for any
     measure $\mu\in\PX$ with $\mu(W)<+\infty$, there is a constant~$C_{\mu}$ for which
     \[
     \forall\, k\geq 1, \quad
     \rho_W\big( \Phi_k(\mu), \mu_1 \big) \leq C_{\mu} \bar{\alpha}^k.
     \]
     Choosing $\mu=\mu_2$ and using the invariance by~$\Phi$ leads to
     \[
     \rho_W ( \mu_2, \mu_1 ) \leq C_{\mu_2} \bar{\alpha}^k.
     \]
     Taking the limit $k\to +\infty$ shows that $\mu_1=\mu_2$, so the invariant measure is
     unique.
\end{proof}

Theorem~\ref{theo:general} also leads to alternative representations of the spectral
radius~$\Lambda$ as a scaled cumulant generating function~\cite{kontoyiannis2005large} and as the
average rate of creation of probability of the dynamics. This is the purpose of the following
result.

\begin{theorem}
  \label{prop:SCGF}
  Let~$Q^f$ be as in Theorem~\ref{theo:general} and define $\lambda=\log(\Lambda)$.
  Then, for any $\mu\in\PX$ with $\mu(W)<+\infty$,
  \begin{equation}
    \label{eq:SCGF}
    \lambda = \underset{k\to+\infty} {\lim}\ \frac{1}{k}\log
    \Big(\mu\big[ (Q^f)^k\ind \big]\Big).
  \end{equation}
  Moreover,
  \begin{equation}
    \label{eq:average}
    \Lambda = \mu_f^\star\big(Q^f\ind \big).
  \end{equation}
\end{theorem}

\begin{proof}
  Considering the operator~$Q_h$ introduced in Lemma~\ref{lem:Qh}, we
  have for any $\mu\in\PX$ with $\mu(W)<+\infty$,
  \[
  \mu\big[ (Q^f)^k\ind \big] = \mu \big( \Lambda^k h Q_h^k h^{-1} \big).
  \]
  Taking the logarithm and dividing by~$k$ leads to
  \[
  \frac{1}{k}\log \mu \big[ (Q^f)^k\ind \big] = \log(\Lambda) +
  \frac{1}{k}\log \mu \big( h Q_h^k h^{-1} \big).
  \]
  Lemma~\ref{lem:Qh} shows that $\mu \big( h Q_h^k h^{-1} \big)$ converges to~$c_{\mu,h}^{-1}$,
  where~$c_{\mu,h}$ is defined in~\eqref{eq:cmuh}. Taking the limit $k\to+\infty$ then leads
  to~\eqref{eq:SCGF}.

  In order to prove~\eqref{eq:average}, we use that~$\mu_f^\star$ is a fixed point of~$\Phi$,
  \textit{i.e.} for any $\varphi\in\Linfty_W(\X)$,
  \[
  \mu_f^\star(\varphi ) = \frac{ \mu_f^\star( Q^f\varphi) }{\mu_f^\star( Q^f \ind )}.
  \]
  Taking $\varphi=h\in\Linfty_W(\X)$ and using $Q^f h = \Lambda h$ we obtain
  \[
  \mu_f^\star(h) =  \frac{ \mu_f^\star( \Lambda h) }{\mu_f^\star( Q^f \ind )},
  \]
  so that $\Lambda=\mu_f^\star( Q^f \ind )$, as claimed.
\end{proof}

Although stated in an abstract setting, Theorem~\ref{prop:SCGF} has a natural
interpretation. If $Q^f=\e^{f}Q$ where~$Q$ is the evolution operator of a Markov
chain $(x_k)_{k\in\N}$ with $x_0\sim\mu$, then~\eqref{eq:SCGF} rewrites
\[
\lambda = \underset{k\to+\infty}{\lim}\, \frac{1}{k}\log\, \E_{\mu}\Big[ \e^{\sum_{i=0}^{k-1}f(x_i)}
  \Big],
\]
which is a standard formula for the scaled cumulant generating function (SCGF, or
logarithmic spectral radius) in large deviations theory~\cite{dembo2010large,kontoyiannis2005large}.
We remind that~$\E_{\mu}$ stands for the expectation
with respect to all trajectories with initial condition distributed according to~$\mu$.
On the other hand,~\eqref{eq:average} means that this SCGF can be expressed as the average rate
of creation of probability of the process under the invariant measure.
In particular, if $Q^f=Q$ is the evolution operator of a Markov chain, $\Lambda =1$ since there is
no creation of probability.
Formula~\eqref{eq:average} does not seem typical in the large deviations literature, but was used
in~\cite{ferre2017error} to quantify the bias arising from discretizing a continuous
Feynman--Kac dynamics.

\begin{remark}
  \label{rem:as}
  It should be clear from the proofs that Assumptions 1 to 3 can be adapted or relaxed
  depending on the context. In particular, we typically consider situations in which
  the state space~$\X$ is (a subset of)~$\R^d$,
  and the transition kernel~$Q^f$ has a transition density $p^f(x,y)>0$ jointly continuous
  in $x,y$. In this case, Assumptions~\ref{as:minogeneral}
  and~\ref{as:regularity} are immediately fulfilled by setting
  $\eta_n(dx) = \ind_{K_n}(x)dx/|K_n|$ for each compact~$K_n$, as we will see in
  Section~\ref{sec:continuous}.
  Similarly, the assumption that~$W\geq 1$ can be
  weakened into:~$W$ is lower bounded by a positive constant on each compact set.

  Another remark of interest is that the regularity condition (Assumption~\ref{as:regularity})
  is not satisfied by Metropolis type kernels~\cite{robert2004monte}, which are therefore
  not covered by our analysis. The obstruction here is that we cannot prove with our techniques that the
  essential spectral radius of~$Q^f$ is~0, because the compactness argument used in the first step of the proof of Lemma~\ref{lem:specQV} in Appendix~\ref{sec:specQV} fails. However, we believe that a finer spectral analysis can cope
  with this situation, see for instance~\cite{wu2004essential} for a careful study of the essential
  spectrum of discrete time Markov chains.

  Let us mention that, in Assumption~\ref{as:lyapunovgeneral}, it seems sufficient
  to suppose that $\gamma_n < \Lambda$ for some~$n\geq 1$ in order to obtain
  that $\theta(Q^f) < \Lambda(Q^f)$ in the proof of Lemma~\ref{lem:specQV}. This is sufficient
  to apply the Krein--Rutman theorem, and to obtain a Lyapunov condition for~$Q_h$
  (see Remark~\ref{rem:Qh} in Appendix~\ref{app:Qh}).
  
  It is also possible to keep track of the constants in the proofs of Lemma~\ref{lem:Qh}
  and Theorem~\ref{theo:general}, like in~\cite{hairer2011yet}, and observe that they
  depend on the assumptions through the coefficients~$\gamma_n$, $b_n$,
  $\alpha_n$, the measures~$\eta_n$ and the function~$W$. More precisely, the constants
  deteriorate when~$\alpha_n$ and~$\eta_n(h)$ are small, and~$\gamma_n$, $b_n$
  and~$\sup_{K_n} W$ are large. Therefore, although the term~$\eta_n(h)$ cannot be controlled
  more explicitly under our assumptions,  it seems possible to optimize
  the final constants in Lemma~\ref{lem:Qh} (and thus in Theorem~\ref{theo:general})
  with respect to the choice of~$n$.

  In order to sketch the role of each assumption in the proofs of the results, we display
  in Figure~\ref{fig:proof} a schematic representation of the arguments. We hope this will
  help adapting our framework to situations where our assumptions are not fulfilled as such.
\end{remark}

\begin{figure}[!h]
\psscalebox{0.9 0.9} % Change this value to rescale the drawing.
           {
             \begin{pspicture}(0,-4.815)(18.8,4.815)
               \definecolor{colour0}{rgb}{0.0,0.6,0.6}
               \definecolor{colour1}{rgb}{0.6,0.0,0.0}
               \rput[bl](10.0,3.585){Minorization condition}
               \rput[bl](4.0,3.585){Lyapunov condition }
               \rput[bl](15.6,1.585){Local regularity }
               \rput[bl](14.4,-0.415){Existence of $\bar{n}$ such }
               \psframe[linecolor=black, linewidth=0.06, dimen=outer, framearc=0.1961375](7.6,4.385)(3.6,2.785)
               \psframe[linecolor=black, linewidth=0.06, dimen=outer, framearc=0.1961375](14.0,4.385)(9.6,2.785)
               \psframe[linecolor=black, linewidth=0.06, dimen=outer, framearc=0.1961375](18.4,2.385)(15.2,0.785)
               \psframe[linecolor=black, linewidth=0.06, linestyle=dashed, dash=0.17638889cm 0.10583334cm, dimen=outer, framearc=0.21805875](17.6,0.385)(14.0,-1.615)
               \rput[bl](9.6,1.585){Zero essential radius}
               \rput[bl](9.6,-0.815){Total cone stability}
               \rput[bl](9.6,-3.215){Positive spectral radius}
               \psframe[linecolor=colour0, linewidth=0.03, linestyle=dashed, dash=0.17638889cm 0.10583334cm, dimen=outer, framearc=0.20097657](13.4,2.385)(9.2,1.185)
               \psframe[linecolor=colour0, linewidth=0.03, linestyle=dashed, dash=0.17638889cm 0.10583334cm, dimen=outer, framearc=0.20097657](12.8,-0.015)(9.2,-1.215)
               \psframe[linecolor=colour0, linewidth=0.03, linestyle=dashed, dash=0.17638889cm 0.10583334cm, dimen=outer, framearc=0.20097657](13.6,-2.415)(9.2,-3.615)
               \psline[linecolor=black, linewidth=0.04, arrowsize=0.05291667cm 2.0,arrowlength=1.4,arrowinset=0.0]{<-}(13.35026,1.7765324)(15.199974,1.7934676)(15.199974,1.7934676)
               \rput[bl](5.2,-0.815){Existence of $h>0$}
               \rput[bl](1.2,-0.815){Stability of $Q_h$}
               \psframe[linecolor=colour1, linewidth=0.06, linestyle=dotted, dotsep=0.10583334cm, dimen=outer, framearc=0.1961375](8.4,-0.015)(4.8,-1.615)
               \psframe[linecolor=colour1, linewidth=0.06, linestyle=dotted, dotsep=0.10583334cm, dimen=outer, framearc=0.1961375](4.0,-0.015)(0.8,-1.615)
               
               \psline[linecolor=black, linewidth=0.04, arrowsize=0.05291667cm 2.0,arrowlength=1.4,arrowinset=0.0]{<-}(14.1,0.3)(13.6,1.03)(13.6,2.785) %(13.6,2.435) (12.8,-0.015)(13.6,1.035)
               %{->}(13.60,1.58)(13.6,1.03)(14.1,0.3)
               \psline[linecolor=black, linewidth=0.04, arrowsize=0.05291667cm 2.0,arrowlength=1.4,arrowinset=0.0]{<-}(9.2,-0.015)(6.8,1.185)(5.2,2.785)
               \psline[linecolor=black, linewidth=0.04, arrowsize=0.05291667cm 2.0,arrowlength=1.4,arrowinset=0.0]{<-}(13.6,-2.815)(14.4,-1.615)(14.4,-1.615)
               \psline[linecolor=colour0, linewidth=0.04, linestyle=dashed, dash=0.17638889cm 0.10583334cm, arrowsize=0.05291667cm 2.0,arrowlength=1.4,arrowinset=0.0]{<-}(8.2,-0.015)(9.2,1.585)(9.2,1.585)
               \psline[linecolor=colour0, linewidth=0.04, linestyle=dashed, dash=0.17638889cm 0.10583334cm, arrowsize=0.05291667cm 2.0,arrowlength=1.4,arrowinset=0.0]{<-}(8.4,-0.615)(9.2,-0.615)(9.2,-0.615)
               \psline[linecolor=colour0, linewidth=0.04, linestyle=dashed, dash=0.17638889cm 0.10583334cm, arrowsize=0.05291667cm 2.0,arrowlength=1.4,arrowinset=0.0]{<-}(7.6,-1.615)(9.2,-2.815)(9.2,-2.815)
               \psline[linecolor=colour1, linewidth=0.06, linestyle=dotted, dotsep=0.10583334cm, arrowsize=0.05291667cm 2.0,arrowlength=1.4,arrowinset=0.0]{<-}(4.0,-0.615)(4.8,-0.615)(4.8,-0.615)
               \psline[linecolor=colour1, linewidth=0.06, linestyle=dotted, dotsep=0.10583334cm, arrowsize=0.05291667cm 2.0,arrowlength=1.4,arrowinset=0.0]{<-}(2.8,-0.015)(4.0,1.385)(4.0,2.785)(4.0,2.785)
               \psline[linecolor=colour1, linewidth=0.06, linestyle=dotted, dotsep=0.10583334cm, arrowsize=0.05291667cm 2.0,arrowlength=1.4,arrowinset=0.0]{<-}(2.0,-0.015)(2.0,2.785)(3.6,4.785)(10.8,4.785)(11.6,4.385)(11.6,4.385)
               \rput[bl](14.4,-0.815){that $\eta_{\bar{n}}(K_{\bar{n}}) >0$}
               \rput[bl](4.0,3.185){Assumption 1}
               \rput[bl](15.6,1.185){Assumption 3}
               \rput[bl](14.4,-1.215){(Lemma 1)}
               \psframe[linecolor=black, linewidth=0.06, dimen=outer, framearc=0.21805875](18.8,-2.815)(15.2,-4.815)
               \rput[bl](10.0,3.185){Assumption 2, Eq. (7)}
               \rput[bl](16.0,-3.615){Irreducibility}
               \rput[bl](16.0,-4.015){Assumption 2}
               \rput[bl](16.0,-4.415){Eq. (8)}
               \psline[linecolor=black, linewidth=0.04, arrowsize=0.05291667cm 2.0,arrowlength=1.4,arrowinset=0.0]{->}(16.8,-2.815)(16.0,-1.615)
               \psline[linecolor=colour0, linewidth=0.04, linestyle=dashed, dash=0.17638889cm 0.10583334cm, arrowsize=0.05291667cm 2.0,arrowlength=1.4,arrowinset=0.0]{->}(15.2,-4.415)(9.2,-4.415)(6.0,-1.615)
               \rput[bl](5.2,-1.215){(Lemma 2)}
               \rput[bl](1.2,-1.215){(Lemma 3)}
               \psline[linecolor=black, linewidth=0.04, arrowsize=0.05291667cm 2.0,arrowlength=1.4,arrowinset=0.0]{->}(6.8,2.785)(7.6,1.985)(9.2,1.985)
               \psframe[linecolor=colour1, linewidth=0.06, dimen=outer, framearc=0.187](2.8,-2.815)(0.0,-4.015)
               \psframe[linecolor=colour1, linewidth=0.06, dimen=outer, framearc=0.187](6.8,-2.815)(3.6,-4.015)
               \rput[bl](4.0,-3.615){Uniqueness of $h$}
               \rput[bl](0.4,-3.5){Theorem 1}
               \psline[linecolor=colour1, linewidth=0.04, linestyle=dotted, dotsep=0.10583334cm, arrowsize=0.05291667cm 2.0,arrowlength=1.4,arrowinset=0.0]{->}(1.2,-1.615)(1.2,-2.815)
               \psline[linecolor=colour1, linewidth=0.04, linestyle=dotted, dotsep=0.10583334cm, arrowsize=0.05291667cm 2.0,arrowlength=1.4,arrowinset=0.0]{->}(3.2,-1.615)(4.0,-2.815)
               \psline[linecolor=colour1, linewidth=0.04, linestyle=dotted, dotsep=0.10583334cm, arrowsize=0.05291667cm 2.0,arrowlength=1.4,arrowinset=0.0]{->}(5.6,-1.615)(4.8,-2.815)
               %\psline[linecolor=black, linewidth=0.04, arrowsize=0.05291667cm 2.0,arrowlength=1.4,arrowinset=0.0]{->}(13.60,1.58)(13.6,1.03)(14.1,0.3)
             \end{pspicture}
           }~  
           \caption{Schematic representation of the arguments used for the proofs of
             Lemmas~\ref{lem:preliminary},~\ref{lem:specQV},~\ref{lem:Qh}, and
             Theorem~\ref{theo:general}. The plain lines correspond to Assumptions 1-3
             pointing towards Lemma~\ref{lem:preliminary} and the key ingredients for
             the proof of Lemma~\ref{lem:specQV}. The dashed lines correspond to the actual proof of
             Lemma~\ref{lem:specQV}. The dotted lines correspond to the elements needed for the
             proof of Lemma~\ref{lem:Qh} and its consequences.}
\label{fig:proof}
\end{figure}
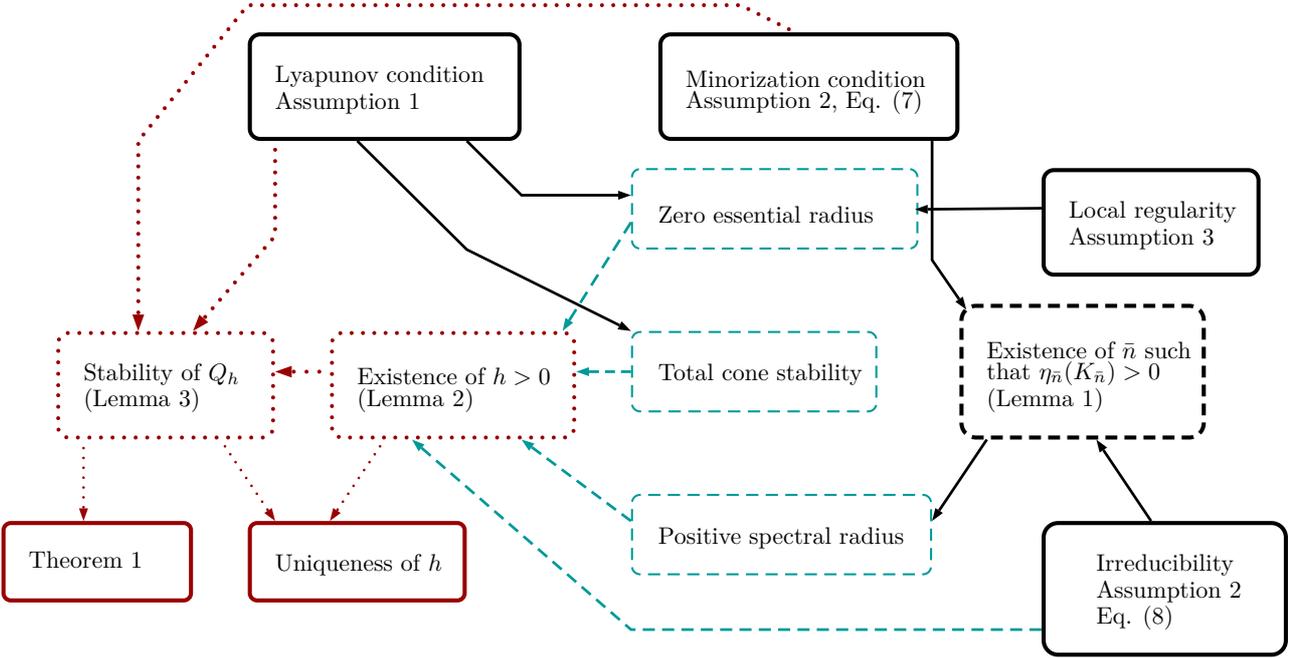

%%%%%%%%%%%%%%%%%%%%%%%%%%%%%%%%%%%%%%%%%%%%%%%%%%%%%%%%%%%%%%%%%%%%%%%%%%%%%%%%%%%%%%%

\subsection{Results in continuous time}
\label{sec:continuous}
Our analysis carries over to time continuous processes, in particular diffusions. In this case,
it is possible to rephrase Assumption~\ref{as:lyapunovgeneral} in terms of the associated
infinitesimal generator. In order to avoid the technical difficulty of dealing with an
infinite dimensional process, we consider a diffusion~$(X_t)_{t\geq 0}$
over $\X= \R^d$ for some integer $d\geq 1$, satisfying the SDE
\begin{equation}
  \label{eq:SDE}
dX_t= b(X_t)\, dt + \sigma(X_t)\, dB_t,
\end{equation}
where $b:\X \to\R^d$, $\sigma :\X \to \R^{d\times m}$ and~$(B_t)_{t\geq 0}$ is
an $m$-dimensional Brownian motion (for some integer $m\geq 1$).
We assume that~$b$ and~$\sigma$ are locally Lipschitz in order for~\eqref{eq:SDE} to have
at least a local solution~\cite[Chapter~XI, Exercice~(2.10)]{revuz2013continuous}.
The noise in the SDE~\eqref{eq:SDE} may be degenerate (\emph{i.e.} $\sigma\sigma^T$ is not of full rank~$d$) provided
some technical conditions are met (see Assumptions~\ref{as:contlyapunov} and~\ref{as:hypoelliptic} below).

The associated infinitesimal generator is given by
\begin{equation}
  \label{eq:gen}
  \Lc = b\cdot \nabla +\frac{\sigma \sigma^T}{2}: \nabla^2
  =  \sum_{i=1}^d b_i \partial_{x_i}
  +\frac{1}{2} \sum_{i,j=1}^d  (\sigma\sigma^T)_{ij}\partial_{x_i}\partial_{x_j}.
\end{equation}
We also consider a measurable function~$f:\X\to\R$ and the corresponding continuous Feynman--Kac
semigroup that reads, for all $t > 0$ and all initial distribution
$\mu\in\PX$,
\begin{equation}
  \label{eq:continuous}
  \Theta_t(\mu)(\varphi) =
  \frac{\E_{\mu}\left( \varphi(X_t)\, \e^{\int_0^t f(X_s)\, ds}\right)}
       {\E_{\mu}\left( \e^{\int_0^t f(X_s)\, ds}\right)}.
\end{equation}
In this setting, we define the operator
\[
\big(P_t^f\varphi\big)(x) = \E_x\left( \varphi(X_t)\, \e^{\int_0^t f(X_s)\, ds}\right),
\]
so that~\eqref{eq:continuous} is the natural continuous counterpart
of~\eqref{eq:dynamics} where, for a fixed time $t>0$, we formally have 
\begin{equation}
  \label{eq:contQP}
Q^f := P_t^f = \e^{t(\Lc + f)}.
\end{equation}
As a result,~$\Theta_t$ satisfies a semigroup property as the discrete time evolution
through~\eqref{eq:onestep}.
In this case, the generator of the weighted evolution operator~$P_t^f$ is $\Lc + f$.
As for the discrete semigroup~\eqref{eq:dynamics}, we are interested in
the long time behavior of quantities such as~\eqref{eq:continuous}. When $b=0$ and
$\sigma=\sqrt{2}\,\id$, decay estimates of~\eqref{eq:continuous} 
in~$L^2(\X)$ towards a well-defined limit can be obtained by considering the spectral
properties of the Schrödinger type operator $- \Delta  - f$, as in~\cite{rousset2006control}.
When $\sigma=\sqrt{2}\,\id$ and $b=-\nabla U$ is the gradient of a potential energy, the
operator $\Lc +f$ is self-adjoint in~$L^2(\e^{-U})$ (see for instance~\cite{bakry2013analysis}),
and the unitary transform $\varphi\mapsto \varphi\,\e^{-\frac{U}{2}}$ leads to an analysis
similar to the Schrödinger case. More precisely, $\Lc +f$ is unitarily equivalent to
\[
\Delta - \frac{1}{4}|\nabla U|^2 +\frac{1}{2} \Delta U + f,
\]
which can be studied by the theory of symmetric operators~\cite{helffer2013spectral}.
In both cases, the operator $\Lc + f$ is self-adjoint on a suitable Hilbert space, so that
the Rayleigh formula can be used. It is also possible to study the spectral properties
of~$P_t^f$ when $b\neq - \nabla U$ and~$\X$ is bounded through the Krein--Rutman theorem (see
\textit{e.g.}~\cite[Proposition~1]{ferre2017error}).
To the best of our knowledge, the case $b\neq -\nabla U$ in an unbounded space~$\X$
remains open in general.

Our analysis provides a practical criterion to study the long time behavior of~\eqref{eq:continuous}
through the Lyapunov function techniques developed in Section~\ref{sec:results}. The continuous
counterpart of Assumption~\ref{as:lyapunovgeneral} can be stated in the following simple form.

\begin{assumption}
  \label{as:contlyapunov}
  Let~$(X_t)_{t\geq 0}$ be the dynamics~\eqref{eq:SDE} with generator~\eqref{eq:gen}. There exists
  a~$C^2(\X)$ function $W:\X\to[1,+\infty)$ going to infinity at infinity such that
    \begin{equation}
      \label{eq:contlyapunov}
      \frac{(\Lc + f ) W}{W} \xrightarrow[|x|\to +\infty]{} - \infty.
    \end{equation}
    In addition, there exist a $C^2(\X)$ function $\mathscr{W}:\X\to[1,+\infty)$ and a
      constant $c \geq 0$ such that
      \begin{equation}
        \label{eq:majoLW}
      \varepsilon(x):= \frac{\mathscr{W}(x)}{W(x)}\xrightarrow[|x|\to +\infty]{} 0, \qquad
       \frac{(\Lc + f) \mathscr{W}}{\mathscr{W}}  \leq c.
      \end{equation}
\end{assumption}

Condition~\eqref{eq:contlyapunov} can be checked by direct computations, as shown on some examples
in Section~\ref{sec:contappli}. Finding a function~$\mathscr{W}$ such
that~\eqref{eq:majoLW} holds is usually done by considering Lyapunov functions in an exponential form, \emph{i.e.} $W(x) = \e^{a U(x)}$ for some function~$U:\X\to\R$ and $a>0$, and~$\mathscr{W}(x) = \e^{a' U(x)}$ for $0<a'<a$. We refer for instance to~\cite[Proposition~1]{ferre2019fine} for precise sufficient conditions for~\eqref{eq:majoLW} to hold in this context. 
In the proof of
Theorem~\ref{theo:continuous},~\eqref{eq:contlyapunov}-\eqref{eq:majoLW} are used to control~$P_t^f$
thanks to a Grönwall lemma. It is also important to remark that, in the case
$f = 0$,~\eqref{eq:contlyapunov} is a standard condition for the ergodicity of SDEs and
compactness of the evolution operator~$P_t$,
see~\cite[Theorem 8.9]{bellet2006ergodic}. As in Section~\ref{sec:discrete}, some regularity
of the transition kernel is required. A natural condition in the context of diffusions reads
as follows~\cite[Section 7]{bellet2006ergodic}.

\begin{assumption}
  \label{as:hypoelliptic}
  The function~$\sigma$ is continuous and, for any~$t>0$, the transition
  kernel~$P_t^f$ has a continuous density~$p_t^f$ with respect to the Lebesgue measure,
  that is
  \[
   \forall\, x,y\in\X, \quad P_t^f(x,dy) = p_t^f(x,y)\,dy.
  \]
  Moreover, it holds
  \[
  \forall\,x,y\in\X,\quad p_t^f(x,y)>0.
  \]
\end{assumption}

This assumption is standard for diffusion processes and, as shown in the
proof of Theorem~\ref{theo:continuous}, it implies Assumptions~\ref{as:minogeneral}
and~\ref{as:regularity} in Section~\ref{sec:discrete}.
It holds true in particular for elliptic diffusions with regular coefficients and additive
noise ($b\in C^{\infty}(\X)$ and $\sigma =\id$). For degenerate diffusions, possibly with
multiplicative noise, this result can be obtained through hypoelliptic conditions and
controllability~\cite{stroock1972support,bellet2006ergodic,wu2001large}. This is explained
in detail in~\cite{ferre2019fine} in the context of large deviations, which allows to perform
a similar spectral analysis as the one performed here.

We now state the continuous version of Theorem~\ref{theo:general}.

\begin{theorem}
  \label{theo:continuous}
  Consider the dynamics~\eqref{eq:continuous} induced by the SDE~\eqref{eq:SDE} and suppose
  that Assumptions~\ref{as:contlyapunov} and~\ref{as:hypoelliptic} hold. Then, there exist
  a unique invariant measure~$\mu_f^\star$ and $\kappa >0$ such that, for any initial
  measure $\mu\in\PX$ with $\mu(W)<+\infty$, there is~$C_{\mu}>0$ for which
  \begin{equation}
    \label{eq:continuouscv}
    \forall\,\varphi\in\Linfty_W(\X),\quad \forall\, t >0, \quad
    \big| \Theta_t(\mu)(\varphi) - \mu_f^\star(\varphi) \big| \leq C_{\mu} \, \e^{-\kappa t}
    \| \varphi \|_{\Linfty_W}.
  \end{equation}
  Moreover, the invariant measure satisfies~$\mu_f^\star(W)<+\infty$
  and $\Theta_t(\mu_f^\star) = \mu_f^\star$ for all~$t\geq 0$.
\end{theorem}

\begin{proof}
  The idea of the proof is to show that, for any $t>0$, the evolution operator
  \[
  \big(P_{t}^f\varphi\big)(x) = \E_x\left[ \varphi(X_{t})\, \e^{\int_0^{t}f(X_s)\,ds}
    \right]
  \]
  satisfies the assumptions of Theorem~\ref{theo:general}.

  \textbf{Step 1: Minorization and regularity.}
  We first show that, by Assumption~\ref{as:hypoelliptic},~$P_t^f$ satisfies
  Assumptions~\ref{as:minogeneral} and~\ref{as:regularity}. A first
  remark is that, since~$P_t^f$ is assumed to have a continuous density with respect to
  the Lebesgue measure, Assumption~\ref{as:regularity} immediately holds.

  It is enough to prove the minorization condition (Assumption~\ref{as:minogeneral})
  for measurable subsets of $\X=\R^d$. Consider the compact
  sets $K_n=B(0,n)$, \textit{i.e.} the balls centered at~$0$ with radius $n\geq 1$. For
  a measurable set $S\subset\R^d$ and $n\geq 1$, we have, for all $x\in K_n$,
  \begin{equation}
    \label{eq:minocont}
  (P_t^f\ind_S)(x)= \int_S p_t^f(x,y)\,dy \geq  \int_{S\cap K_n} p_t^f(x,y)\,dy \geq
  \Big( \inf_{x,y\in K_n}\, p_t^f(x,y)\Big)\ | S\cap K_n|,
  \end{equation}
  where we denote by~$|A|$ the Lebesgue measure of a measurable set $A\subset \R^d$. 
  As a result,~\eqref{eq:minogeneral} holds for all $n\geq 1$ with
  \[
  \eta_n(S) = \frac{|S\cap K_n|}{|K_n|},\qquad \alpha_n  = |K_n|\Big( \inf_{x,y\in K_n} p_t^f(x,y)\Big) >0. 
  \]
  Finally, let us check that~\eqref{eq:irreducibility} is satisfied. Take $\varphi\in\Linfty_W(\X)$
  with $\varphi\geq 0$ such that
  \[
  \eta_n(\varphi)=\frac{1}{|K_n|}\int_{K_n}\varphi(x)\,dx = 0,
  \]
  for any $n\geq n_0$ for an arbitrary $n_0\geq 1$. Since for any compact set~$K\subset\X$
  there exists~$m\geq 1$ such that~$K\subset K_m$, this implies that $\varphi=0$
  almost everywhere, so $Q^f\varphi =0$ everywhere since~$Q^f$ has a continuous
  density with respect to the Lebesgue measure. Therefore,
  Assumption~\ref{as:minogeneral} is satisfied.

  \textbf{Step 2: Lyapunov condition.}
  Let us now show that Assumption~\ref{as:lyapunovgeneral} holds.  First,
  Assumption~\ref{as:contlyapunov} is equivalent to the existence of positive
  sequences~$(a_n)_{n\in\N}$, $(b_n)_{n\in\N}$ such that
  \begin{equation}
    \label{eq:lyapunovab}
  (\Lc + f ) W \leq -a_n W + b_n,
  \end{equation}
  with $a_n\to +\infty$ as $n\to +\infty$. We then compute, for
  any $t>0$ and $n\in\N$,
  \begin{equation}
    \label{eq:dPtf}
  \frac{d}{dt}\left( \e^{a_n t} P_t^f W\right)= \e^{a_n t} P_t^f
  \big( a_n W + (\Lc + f) W  \big)
  \leq b_n \e^{a_n t} P_t^f \ind.
  \end{equation}
  We can now bound the right hand side of the above expression using~\eqref{eq:majoLW}.
  Since~$\mathscr{W}\geq 1$,
    \begin{equation}
      \label{eq:Ptf}
  \big(P_t^f\ind\big) (x) = \E_x\left[ \e^{\int_0^t f(X_s) \, ds} \right]\leq
  \E_x\left[ \mathscr{W}(X_t)\, \e^{\int_0^t f(X_s) \, ds} \right].
  \end{equation}
  From the second condition in~\eqref{eq:majoLW},~\eqref{eq:Ptf} becomes
  \[
  \big( P_t^f\ind\big) (x) \leq \e^{ct}\, \E_x\left[ \mathscr{W}(X_t)\,
    \e^{-\int_0^t \frac{\Lc \mathscr{W}}{\mathscr{W}}(X_s) \, ds} \right].
  \]
  Inspired by a similar calculation in~\cite{wu2001large}, we see that the right
  hand side of the above equation is a supermartingale. Indeed, introducing
  \[
  M_t =  \mathscr{W}(X_t)\,\e^{-\int_0^t  \frac{\Lc \mathscr{W}}{\mathscr{W}}(X_s) \, ds},
  \]
  Itô formula shows that
  \[
  d M_t = \e^{-\int_0^t \frac{\Lc \mathscr{W}}{\mathscr{W}}(X_s) \, ds}\,
  \nabla \mathscr{W}^T(X_t) \sigma(X_t) \,dB_t,
  \]
  so that $M_t$ is a local martingale (using that~$\sigma$ and~$b$ are locally Lipschitz hence
  continuous and~$\mathscr{W}$ has a continuous derivative,
  see~\cite[Chapter~3, Proposition~2.24]{karatzas2012brownian}).
  %We have a local martingale, hence a martingale for a sequence of stopping times
  % going to infinity. Use Fatou lemma's to invert the limit of the stopping times
  % and the expectation to obtain the supermartingale property. 
  Since~$M_t$ is nonnegative, it is a supermartingale by Fatou's lemma. As a result,
  $\E_x[M_t]\leq M_0 = \mathscr{W}(x)$. The inequality~\eqref{eq:Ptf} then becomes
  \[
  \big(P_t^f\ind\big) (x) \leq \e^{ct}\, \E_x\left[ M_t \right]\leq  \,\e^{ct}\, \mathscr{W}(x).
  \]

  Coming back to~\eqref{eq:dPtf}, we obtain
  \[
  \frac{d}{dt}\left( \e^{a_n t} P_t^f W\right)
  \leq  b_n \e^{(a_n +c) t}\, \mathscr{W}.
  \]
  Integrating in time,
  \[
  \big( \e^{a_n t}P_t^f W - W\big)(x) \leq  b_n \frac{ \e^{(a_n + c)t}}{a_n + c}\mathscr{W}(x).
  \]
  As a result
  \begin{equation}
    \label{eq:Ptfan}
  P_t^f W(x) \leq \widetilde{\gamma}_n W(x) + c_n\mathscr{W}(x),
  \end{equation}
  with
  \[
  \widetilde{\gamma}_n = \e^{-a_n t}, 
  \quad c_n = \frac{ b_n\, \e^{c t} }{a_n + c}\geq 0.
  \]
  At this stage,~\eqref{eq:lyapunovgeneral} holds with the indicator function replaced
  by the function~$\mathscr{W}$. However, using the first condition in~\eqref{eq:majoLW},
  we can find a compact set~$K_n$ such that $c_n \varepsilon(x) \leq \widetilde{\gamma}_n$
  outside~$K_n$. Using this set and~$\mathscr{W} = \varepsilon W$,~\eqref{eq:Ptfan} becomes
  \[
  \begin{aligned}
  P_t^f W(x)  & \leq \widetilde{\gamma}_n W(x) + c_n \ind_{K_n}(x) \mathscr{W}(x)
  +  c_n \varepsilon(x) W(x) \ind_{K_n^c}(x)
  \\ &\leq 2 \widetilde{\gamma}_n W(x) +  c_n\left( \sup_{K_n}\mathscr{W}\right) \ind_{K_n}(x).
  \end{aligned}
  \]
  Setting $\gamma_n = 2 \widetilde{\gamma}_n$ and $b_n= c_n \sup_{K_n}\mathscr{W}$, we see
  that 
  \begin{equation}
    \label{eq:Ptfbound}
  P_t^f W \leq \gamma_n W  +  b_n \ind_{K_n},
  \end{equation}
  with $\gamma_n\to 0$ as $n\to +\infty$. This means that $P_t^f$ satisfies
  Assumption~\ref{as:lyapunovgeneral}, and hence fullfils
  all the assumptions of Theorem~\ref{theo:general}.

  \textbf{Step 3: using Theorem~\ref{theo:general}.}
  We now use that~$P_t^f$ satisfies the assumptions of Theorem~\ref{theo:general}
  to conclude the proof. Fix $t_0>0$. There exist a unique
  measure~$\mu_{f,t_0}^\star$ and a constant $\kappa_{t_0} >0$ such that for any $\mu\in\PX$ with
  $\mu(W)<+\infty$, it holds (with the constant $C_{\mu}>0$ defined in~\eqref{eq:Cmulast})
  \[
   \forall\, \varphi\in\Linfty_W(\X), \quad 
 \forall\, k\geq 1,\qquad
  \left| \frac{\mu\big( (P_{t_0}^f)^k\varphi \big)}{\mu\big( (P_{t_0}^f)^k\ind \big)} 
  - \mu_{f,t_0}^\star(\varphi) \right| \leq C_{\mu}\, \e^{-k \kappa_{t_0}} \| \varphi \|_{\Linfty_W}.
  \]
  We next show that~\eqref{eq:continuouscv} can be obtained for any $t>0$ (and not
  only multiples of~$t_0$) and that the invariant measure~$\mu_{f,t_0}^\star$ actually does
  not depend on~$t_0$. This follows by a standard time decomposition
  argument~\cite{lelievre2016partial,herzog2017ergodicity}. Indeed, for
  any $t>0$, we set $t = k t_0 + r$ with $r \in [0, t_0)$, and we use the semigroup property to obtain
    \[
    \Theta_t(\mu)(\varphi)= \Theta_{k t_0}\left( \frac{\mu P_r^f}{\mu(P_r^f\ind)}\right)(\varphi)
    = \frac{\mu_r\big( (P_{t_0}^f)^k\varphi \big)}{\mu_r\big( (P_{t_0}^f)^k\ind \big)},
    \]
    where we defined $\mu_r$ as
    \[
    \mu_r(\varphi)=\frac{\mu( P_r^f\varphi) }{\mu(P_r^f\ind)}.
    \]
    We then only need to control the familly of initial distributions~$(\mu_r)_{r\in[0,t_0)}$.
    Step~$1$ in the proof shows that~$\mu(P_r^f\ind) >0$ (using~\eqref{eq:minocont}).
    Then, in view of~\eqref{eq:Ptfbound}, the evolution operator~$P_r^f$
    maps~$\Linfty_W(\X)$ to~$\Linfty_W(\X)$ for any $r >0$, so $\mu_r(W)<+\infty$ and
    thus~$\mu_r$ defines an admissible initial condition in Theorem~\ref{theo:general}. This leads to:
    \begin{equation}
      \label{eq:intconvtheta}
    \forall\,\varphi\in\Linfty_W(\X),\quad \forall\, t >0, \quad
    \big| \Theta_t(\mu)(\varphi) - \mu_{f,t_0}^\star(\varphi) \big| \leq
    \left(\underset{r\in[0,t_0)} {\sup} C_{\mu_r}\right) \, \e^{-\kappa_{t_0} \frac{t}{t_0}}
      \| \varphi \|_{\Linfty_W},
    \end{equation}
    where the constant~$C_{\mu_r}$ is given in~\eqref{eq:Cmulast}. In view
    of~\eqref{eq:Cmulast}, it remains to bound
    \begin{equation}
      \label{eq:supCmu}
    \underset{r\in[0,t_0)} {\sup}\ \frac{\mu_r(W)}{\mu_r(h)} = \underset{r\in[0,t_0)} {\sup}\
        \frac{\mu(P_r^f W)}{\mu(P_r^f h_{t_0})},
    \end{equation}
    where~$h_{t_0}$ is the principal eigenvector associated to~$P_{t_0}^f$
    with eigenvalue~$\Lambda_{t_0}$ (using Lemma~\ref{lem:specQV}). The numerator in the latter
    expression is easily bounded uniformly in~$r$ using~\eqref{eq:Ptfbound}. Standard semigroup
    analysis shows that $h_{t_0}=h$ does not depend on~$t_0$ and~$\Lambda_{t_0}=\e^{t_0\alpha}$
    for some~$\alpha\in\R$. Therefore, for any~$r\in [0,t_0)$,
    $P_r^f h_{t_0} = \e^{r \alpha}h$, and the denominator in~\eqref{eq:supCmu} is bounded
    away from~$0$ independently on~$r$.

    We finally prove that the invariant measure~$\mu^\star_{f,t_0}$ does not depend on~$t_0$.
    Following the same procedure for another time~$t_1 >0$ shows
    that~\eqref{eq:intconvtheta} holds with an invariant measure~$\mu_{f,t_1}^\star$. Then,
    for any~$\varphi\in\Linfty_W(\X)$, $\mu\in\PX$ with $\mu(W)<+\infty$ and~$t>0$ we have
        \[
        \begin{aligned}
          \big| \mu_{f,t_0}^\star(\varphi) - \mu_{f,t_1}^\star(\varphi)\big| & \leq
          \big| \Theta_t(\mu)(\varphi) - \mu_{f,t_0}^\star(\varphi) \big|
          +     \big| \Theta_t(\mu)(\varphi) - \mu_{f,t_1}^\star(\varphi) \big|
        \\ & \leq
        \left(\underset{r\in[0,t_0)} \sup C_{\mu_r}\right) \, \e^{-\kappa_{t_0} \frac{t}{t_0}}
          \| \varphi \|_{\Linfty_W}
          + \left(\underset{r\in[0,t_1)} \sup C_{\mu_r}\right) \, \e^{-\kappa_{t_1} \frac{t}{t_1}}
            \| \varphi \|_{\Linfty_W}.
        \end{aligned}
      \]
      Taking the limit~$t\to +\infty$ on the right hand side shows
      that~$\mu_{f,t_0}^\star=\mu_{f,t_1}^\star$, so the invariant measure is independent of the
      arbitrary time~$t_0$. This concludes the proof of Theorem~\ref{theo:continuous}.
\end{proof}

We close this section by mentioning that, under the assumptions of Theorem~\ref{theo:continuous},
it is also possible to define the logarithmic spectral radius of the dynamics
as in Theorem~\ref{prop:SCGF}, which reads in this case
\[
\lambda = \underset{t\to+\infty}{\lim}\, \frac{1}{t} \log\, \E_{\mu}\Big[ \e^{\int_0^t f(X_s)\,ds}
  \Big],
\]
for any initial measure~$\mu$ that satisfies $\mu(W)<+\infty$.
We do not reproduce the proof of this result which is similar to that of Theorem~\ref{prop:SCGF},
and refer to~\cite{ferre2019fine} for more results on the cumulant function~$\lambda$.

%%%%%%%%%%%%%%%%%%%%%%%%%%%%%%%%%%%%%%%%%%%%%%%

\section{Applications}
\label{sec:applications}

Since our study was first motivated by practical situations, we provide in this section a
number of finite dimensional examples where our framework provides simple criteria for proving
convergence of the Feynman--Kac semigroup towards an invariant measure.
Sections~\ref{sec:discrappli} and~\ref{sec:contappli} are concerned with discrete and
continuous time applications respectively.
Section~\ref{sec:discretization} presents a convergence result for numerical discretizations
of~\eqref{eq:continuous}, where convergence rates are uniform in the time step.

\subsection{Examples in discrete time}
\label{sec:discrappli}
In this section, we provide two typical examples of Markov chains for which our results apply.
First of all, let us consider the Diffusion Monte Carlo case where $f = -V$
and~$V$ stands for a Schrödinger potential.

\begin{prop}
  \label{prop:DMC}
  Consider a weighted evolution operator $Q^V=\e^{-V}Q$ in $\X=\R^d$ with Gaussian increments
  $Q(x,dy)=(2\pi\sigma^2)^{-\frac{d}{2}}\e^{- \frac{(x-y)^2}{2\sigma^2}}dy$, and where~$V$ is a continuous function.
  Then, if $V(x)\to+\infty$ when $|x|\to +\infty$, $W(x)=\ind$ is a Lyapunov function for~$Q^V$
  in the sense of Assumption~\ref{as:lyapunovgeneral}.
  Moreover, if there exist constants $a>0$ and $c\in\R$ such that 
  \begin{equation}
    \label{eq:Vmino}
    V(x)\geq a |x|^2 - c,
  \end{equation}
  then $W(x)=\e^{\beta x^2}$ is a Lyapunov function for
  \[
  0 < \beta < \frac{a}{2}\left( \sqrt{1+\frac{2}{a \sigma^2}} -1\right).
  \]
  Finally, Assumptions~\ref{as:minogeneral} and~\ref{as:regularity} hold true, so that
  Theorem~\ref{theo:general} applies for these choices of Lyapunov function.
\end{prop}

The interpretation of this result is the following. In the Diffusion Monte Carlo setting,
the confinement cannot be provided by the dynamics, since it is a Gaussian random walk over~$\R^d$.
However, the external potential~$V$ gives a small weight to the trajectories going to infinity,
which makes the dynamics stable. If more information is available on the growth of~$V$,
we obtain better integrability results for the invariant measure~$\mu_V^\star$ through Lyapunov
functions growing faster at infinity.

\begin{proof}
  Let us first check that $W= \ind$ is a Lyapunov function when~$V$ goes to infinity
  at infinity. Note that, for any compact set $K\subset\R^d$,
  \[
  \big(Q^V\ind\big)(x) = \e^{-V(x)}= \ind_{K^c}(x)\, \e^{-V(x)} + \ind_{K}(x)\, \e^{-V(x)}.
  \]
  Taking an increasing sequence of compact sets~$K_n$ (in the sense of inclusion) and
  setting $\gamma_n= \sup_{K_n^c}\e^{-V}$, $b_n=\sup_{K_n}\e^{-V} < +\infty$, we obtain
  \[
  Q^V\ind \leq \gamma_n \ind + b_n \ind_{K_n},
  \]
  which proves the first assertion since $\gamma_n\to 0$ as $n\to + \infty$.

  Let us now assume that~\eqref{eq:Vmino} holds.
  Setting $W(x)=\e^{\beta x^2}$, under the condition
  \begin{equation}
    \label{eq:condbeta}
  \beta < \frac{1}{2\sigma^2},
  \end{equation}
  an easy computation shows that
  \[
  QW(x) = \frac{\e^{\frac{\beta}{1- 2 \beta \sigma^2}x^2}}{(1 - 2\beta\sigma^2)^{\frac{d}{2}}}.
  \]
  We remark that~$W$ is not a Lyapunov function for~$Q$ since $1- 2 \beta \sigma^2 < 1$.
  However, setting \[
  C_d=(1 - 2\beta\sigma^2)^{-\frac{d}{2}},
  \]
  we have
  \[
  Q^VW(x) =C_d\, \e^{-V(x) + \frac{\beta}{1- 2 \beta \sigma^2}x^2}
  \leq C_d \, \e^{c - a x^2 + \frac{\beta}{1- 2 \beta \sigma^2}x^2 - \beta x^2}W(x)
  = C_d'\, \e^{ - a x^2 + \frac{2\beta^2\sigma^2}{1- 2 \beta \sigma^2}x^2}W(x),
  \]
  with $C_d'=C_d\,\e^c$. One can then check that the choice
  \begin{equation}
    \label{eq:condbetaa}
  0 < \beta < \frac{a}{2}\left( \sqrt{1+\frac{2}{a \sigma^2}} -1\right)
  \end{equation}
  leads to
  \[
  - a + \frac{2 \beta^2\sigma^2}{1- 2 \beta \sigma^2} < 0.
  \]
  Note that, since
  \[
  \frac{a}{2}\left( \sqrt{1+\frac{2}{a \sigma^2}} -1\right) < \frac{1}{2\sigma^2},
  \]
  the condition~\eqref{eq:condbeta} is automatically satisfied when~$\beta$ is chosen according
  to~\eqref{eq:condbetaa}. Next, when~$\beta$ satisfies~\eqref{eq:condbetaa}, the
  function
  \[
  \varepsilon(x)= \e^{ - a x^2 + \frac{\beta^2\sigma^2}{1- 2 \beta \sigma^2}x^2}
  \]
  tends to zero at infinity. Therefore, taking increasing compact sets $K_n$ (such as balls
  of increasing radii),
  \[
  (Q^V W)(x) = \ind_{K_n^c}(x) \varepsilon(x) W(x) + \ind_{K_n}(x) \varepsilon(x) W(x)
  \leq \gamma_n W(x) + b_n \ind_{K_n}(x),
  \]
  with $\gamma_n = \sup_{K_n^c} \varepsilon \to 0$ as $n\to+\infty$ and
  $b_n = \sup_{K_n} \varepsilon W<+\infty$.
  Hence~$W$ is a Lyapunov function for~$Q^V$ for this choice of~$\beta$, \textit{i.e.}
  Assumption~\ref{as:lyapunovgeneral} is satisfied.

  Assumption~\ref{as:regularity} is easily seen to hold. It therefore suffices to prove
  the minorization condition (Assumption~\ref{as:minogeneral}). Take a compact set~$K$ with
  non zero Lebesgue measure, and let us first show
  that the condition of Assumption~\ref{as:minogeneral} holds for~$Q$. It is enough to
  prove the condition for the indicator function of any borel set $S\subset \X$. 
  Denoting by $D_K=\sup \{|x-y|,\, x\in K, \, y\in K\}$ the diameter of~$K$,
  we compute for any $x\in K$
  \[\setlength{\jot}{10pt}
  \begin{aligned}
    (Q\ind_S)(x) = Q(x,S) & = \int_S \e^{ - \frac{ (x-y)^2}{2\sigma^2}}\, dy
      \geq \int_{S\cap K} \e^{ - \frac{ (x-y)^2}{2\sigma^2}}\, dy
      \geq \inf_{x\in K}\, \int_{S\cap K} \e^{ - \frac{ (x-y)^2}{2\sigma^2}}\, dy
     \\
    &  \geq \e^{ - \frac{ D_K^2}{2\sigma^2}}   \int_{S\cap K}\, dy
     \geq \e^{ - \frac{ D_K^2}{2\sigma^2}}\,   |S\cap K|,
  \end{aligned}
  \]
  where we denote again by~$|A|$ the Lebesgue measure of a measurable set $A\subset \R^d$.
  This motivates defining
  \[
  \alpha_K = \e^{ - \frac{ D_K^2}{2\sigma^2}} |K|>0, \quad
  \eta_K(S) = \frac{|S\cap K| }{|K|}.
  \]
  Note also that, since $|K| \in ( 0,+\infty)$, $\eta_K$ is a probability measure. Finally,
  since~$V$ is continuous,
  \[
  \forall\, x\in K, \quad Q^V(x , \cdot) \geq \alpha_V \eta_K(\cdot),
  \]
  with $\alpha_V = \alpha_K\, \e^{ - \sup_K V} > 0$. Choosing $K_n=B(0,n)$ the centered
  balls of radius~$n$, we see that~\eqref{eq:irreducibility} holds using arguments similar 
  to the ones used for the proof of Theorem~\ref{theo:continuous},
  hence $Q^V$ satisfies Assumption~\ref{as:minogeneral}.
\end{proof}

We now provide an example where the dynamics~$Q$ admits a Lyapunov function~$W$ in the
sense of the condition~\eqref{eq:lyapunov} recalled in Appendix~\ref{sec:tools}, and this
function is also a Lyapunov function for~$Q^f$ when~$f$ does not grow too fast.

\begin{prop}
  Consider the dynamics corresponding to a discrete Ornstein--Uhlenbeck process
  in~$\R^d$, namely
  \[
  x_{k+1}= \rho x_k + \sigma G_k,
  \]
  where $\rho\in(-1,1)$, $\sigma \in\R$ and~$(G_k)_{k\geq1}$ is a familly of independent
  standard $d$-dimensional
  Gaussian random variables. Define the operator $Q^f=\e^{f}Q$ with~$f$ a continous
  function such that there exist constants $a>0$, $c\geq 0$, $0\leq p < 2$
  for which $f(x) \leq  a |x|^p + c$.

  Then, the Feynman--Kac dynamics associated to~$Q^f$ satisfies the assumptions of
  Theorem~\ref{theo:general} with Lyapunov function $W(x) = \e^{\beta x^2}$ when
  \[
  0 < \beta < \frac{1-\rho^2}{2\sigma^2}.
  \]
\end{prop}

The interpretation of this result is quite different from the interpretation of
Proposition~\ref{prop:DMC}. Here, the confinement is provided by the
dynamics itself, and the weight~$f$ has to be controlled by the Lyapunov function of the dynamics.
In that case it is important to find a <<strong enough>> Lyapunov function in order for this
control to be possible. Quite typically, if~$f$ is unbounded, $W(x)=x^2$ is a Lyapunov
function for~$Q$, but not for~$Q^f$. On the other hand, if~$f$ is bounded above, the result is
straightforward.

\begin{proof}
  We set $W(x)=\e^{\beta x^2}$ and first compute
  \[
  QW(x)  = \E \left[ W(x_{k+1}) \, \big| \, x_k = x 
    \right]  = \E_G \left[ \e^{ \beta | \rho x + \sigma G |^2}
    \right]
  = \e^{\beta \rho^2 x^2} \E_G \left[
    \e^{\beta ( 2\sigma\rho x G + \sigma^2 G^2)}
    \right].
  \]
  For $\beta < 1/(2\sigma^2)$, an easy computation similar to that of
  Proposition~\ref{prop:DMC} shows that
  \[
  QW(x)=\frac{1}{(1-2\beta \sigma^2)^{\frac{d}{2}}}\e^{\frac{\rho^2}{1-2\beta\sigma^2} \beta x^2}.
  \]
  Define now
  \[
  \delta_{\beta}=\frac{\rho^2}{1-2\beta\sigma^2}.
  \]
  Then $\delta_{\beta} \in (0, 1)$ and $1 - 2\beta\sigma^2 >0$ when
  \[
  \beta \in \Big( 0, \frac{1-\rho^2}{2\sigma^2}\Big).
  \]
  This leads to
  \[
  \e^{f(x)}QW(x) = \frac{1}{(1-2\beta \sigma^2)^{\frac{d}{2}}} \e^{f(x) + (\delta_{\beta}
    -1 )x^2} W(x)
  \leq  \frac{1}{(1-2\beta \sigma^2)^{\frac{d}{2}}} \e^{a|x|^p + c + (\delta_{\beta}
    -1 )x^2} W(x)
  = \varepsilon(x) W(x),
  \]
  with $\varepsilon(x)\to 0$ as $|x|\to + \infty$. Therefore, by considering
  againg $K_n=B(0,n)$, we see that
  \[
  \big(Q^f W\big)(x) = \ind_{K_n^c}(x) \varepsilon(x) W(x) +\ind_{K_n}(x) \varepsilon(x) W(x) 
  \leq \gamma_n W(x) + \ind_{K_n}(x) b_n,
  \]
  where $\gamma_n = \sup_{K_n^c} \varepsilon \to 0$ as $n\to + \infty$,
  and $b_n=\sup_{K_n} \varepsilon \,  W<+\infty$. This shows that
  Assumption~\ref{as:lyapunovgeneral} is satisfied. Assumptions~\ref{as:minogeneral}
  and~\ref{as:regularity} follow by arguments similar to those used in the proof of
  Proposition~\ref{prop:DMC}.
\end{proof}

The latter examples do not intend to form a complete overview of the possible practical
cases. However, they seem characteristic of two typical situations: one where the
confinement comes from the potential $V=-f$, and another one where it
arises from the dynamics. These two strategies correspond respectively to a Diffusion Monte Carlo
context~\cite{hairer2014improved} and a Large Deviations
context~\cite{kontoyiannis2005large}. They are both encoded in the
condition~\eqref{eq:lyapunovgeneral}.

%%%%%%%%%%%%%%%%%%%%%%%%%%%%%%%%%%%%%%%%%%%%%%%%%%%%%%%%%%%%%%%%%

\subsection{Applications to diffusion processes}
\label{sec:contappli}

We now provide some examples where the conditions of Section~\ref{sec:continuous} are met.
Our main concern is the Lyapunov condition, Assumption~\ref{as:contlyapunov}, so we
assume~$f$ and the coefficients of the SDE~\eqref{eq:SDE} to be regular enough for
Assumption~\ref{as:hypoelliptic} to be satisfied. Let us start with a reversible diffusion.

\begin{prop}
  \label{prop:difflyapunov}
  Consider a diffusion process~$(X_t)_{t\geq 0}$ over~$\R^d$ satisfying~\eqref{eq:SDE}
  with $\sigma = \sqrt{2}\,\id$, and assume that the drift
  is given by $b=-\nabla U$, where $U:\X\to \R$ is a smooth potential such that
  $U(x)\to +\infty$ as $|x|\to +\infty$. Assume moreover that~$U$ satisfies
  \begin{equation}
    \label{eq:gradlapl}
    \lim_{|x|\to +\infty} \ \frac{ | \nabla U (x)|^2}{|\Delta U(x)|} = +\infty,
  \end{equation}
  and there exists $1/2 < \beta <1$ such that
  \begin{equation}
    \label{eq:confischro}
  \underset{|x|\to + \infty}{\lim}\  \Big(
  - \beta (1 - \beta) |\nabla U|^2 + \beta\Delta U  + f
  \Big) = - \infty.
  \end{equation}
  Then Assumption~\ref{as:contlyapunov} holds for the Lyapunov function~$W(x) = \e^{\beta U (x)}$. 
\end{prop}

The conditions~\eqref{eq:gradlapl} and~\eqref{eq:confischro} are satisfied for instance for
potentials~$x\mapsto U(x)$ behaving at infinity as~$|x|^q$ with $q > 1$, and weight
functions~$f$ such that $f(x)/|x|^{2(q-1)} \to 0$ as $|x|\to+\infty$. 

\begin{proof}
  The proof follows by simple computations. Indeed, it holds
  \[
  \Lc W  = -\beta \nabla U \cdot (\nabla U) W + \beta\nabla\cdot[ ( \nabla U) W ] 
  = -\beta | \nabla U|^2 W + \beta W \Delta U + \beta^2 |\nabla U |^2 W,
  \]
  so that
  \begin{equation}
    \label{eq:LfWint}
  (\Lc + f) W = \Big( -\beta (1 - \beta ) |\nabla U|^2 +\beta \Delta U + f
  \Big)W, 
  \end{equation}
  hence~\eqref{eq:contlyapunov} in Assumption~\ref{as:contlyapunov} is satisfied.
  The conditions in~\eqref{eq:majoLW} are obtained setting
  \[
  \mathscr{W}(x)= \e^{\theta U(x)},
  \]
  for some~$\theta\in(1/2,\beta)$. It is clear that~$\mathscr{W}/W$ goes to zero at
  infinity, so the first condition in~\eqref{eq:majoLW} holds true. 
  The key remark is then to note that for our choice of~$\theta,\beta$, we have
  \[
   \beta ( 1 - \beta) \leq \theta (1 - \theta) .
  \]
  Therefore,~\eqref{eq:gradlapl} and~\eqref{eq:confischro} show that there
  exist~$c,c'\geq 0$ such that
  \[
  f \leq 
   \beta (1 - \beta) |\nabla U|^2 - \beta\Delta U 
   +c \leq 
  \theta (1 - \theta) |\nabla U|^2 - \theta\Delta U  
   + c' = - 
  \frac{\Lc \mathscr{W}}{\mathscr{W}} + c'.
  \]
  This proves that the second condition in~\eqref{eq:majoLW} holds, which concludes the proof.
\end{proof}

Let us mention that the
conditions in Proposition~\ref{prop:difflyapunov} are similar to conditions appearing in works
on Poincaré inequalities (see~\cite{bakry2008simple} and references therein), and correspond to the
case where the confinement comes from the potential~$U$,~$f$ being a perturbation that
should not go too fast to~$+\infty$ with respect to~$U$.

\begin{remark}
  Proposition~\ref{prop:difflyapunov} is also related to confinement
  conditions for Schrödinger operators. Indeed, using the parameters of
  Proposition~\ref{prop:difflyapunov}, the dynamics is reversible with respect to
  the measure~$\e^{-U}$ and, as noted in Section~\ref{sec:continuous}, it is
  possible to turn the diffusion operator~$\Lc$ into a Schrödinger operator using the
  unitary transform:
  \[
  \Lc \to \e^{-\frac{U}{2}} \Lc \e^{\frac{U}{2}}.
  \]
  Using this transformation, $\Lc + f $ is unitarily equivalent~\cite{lelievre2016partial} to
  the following Schrödinger operator: 
  \[
  \Delta - \frac{1}{4}|\nabla U|^2 +\frac{1}{2} \Delta U + f.
  \]
  We then notice that the confinement condition for this Schrodinger operator is
  precisely~\eqref{eq:confischro} for the limit value $\beta = 1/2$. This shows
  that our Lyapunov condition~\eqref{eq:contlyapunov} is a natural extension of this condition for
  non-reversible dynamics. As a side product, it shows that a slightly modified confinement
  condition for a Schrödinger operator does not only provide convergence in $L^2$-norm, but also
  in a weighted uniform norm, which does not seem to be a standard result. Let us however conclude this remark by pointing out
  that we do not claim here that the conditions~\eqref{eq:gradlapl}-\eqref{eq:confischro} are optimal in any way,
  but they are for sure reminiscent of (admittedly restrictive yet typical) sufficient conditions for Poincar\'e inequalities to hold.
\end{remark}

In the non-reversible setting one cannot hope for a Schrödinger representation,
and the Lyapunov function framework shows its usefulness.
Let us present such an application, drawn from~\cite{eberle2016quantitative},
where the drift behaves polynomialy at infinity.

\begin{prop}
  \label{prop:practicalcontinuous}
  Let~$(X_t)_{t\geq 0}$ satisfy the SDE~\eqref{eq:SDE} with $\sigma =\sqrt{2}\,\id$ and
  where the drift $b$ is such that
  there exist $q > 1$, $\delta >0$, $R>0$ for which
  \begin{equation}
    \label{eq:driftq}
\forall\, |x|\geq R,\quad  b(x) \cdot x \leq - \delta |x|^q.
  \end{equation}
  Assume also that~$f$ is smooth and satisfies $f(x) \leq  a | x |^p$
  for $|x|\geq R$ and some $p < 2q - 2$. Then, Assumption~\ref{as:contlyapunov} holds for the
  Lyapunov function
  \begin{equation}
    \label{eq:deltabeta}
  W(x)=\e^{\beta |x|^q}, \quad with\quad 0 < \beta  < \frac{\delta}{q} .
  \end{equation}
\end{prop}

\begin{proof}
  Setting $W(x) = \e^{\beta |x|^q}$, a simple computation shows that
  \begin{equation}
    \label{eq:calculusW}
  \begin{aligned}
  \Lc W (x) & = \beta q b(x)\cdot x |x|^{q-2} W(x) + \beta q \nabla\cdot (x|x|^{q-2}W(x))
  \\ & = \beta q b(x)\cdot x |x|^{q-2} W(x) + \beta q d|x|^{q-2} W(x) + \beta q (q-2)
  |x|^{q-2}W(x) + \beta^2 q^2 |x|^{2q-2} W(x),
  \end{aligned}
  \end{equation}
  so
  \[
  \frac{\Lc W}{W}(x) = \beta q b(x)\cdot x |x|^{q-2} + \beta q (q + d -2)
  |x|^{q-2} + \beta^2 q^2 |x|^{2q-2}.
  \]
  Using~\eqref{eq:driftq} and the bound on~$f$ leads to, for~$|x|\geq R$, 
  \begin{equation}
    \label{eq:boundLWf}
    \frac{\Lc W}{W}(x) + f(x) \leq  - \beta q (\delta - \beta q)  |x|^{2q-2}
    + \beta q (q + d - 2) |x|^{q-2} + a|x|^p. 
  \end{equation}
  Since~$p<2q-2$,~\eqref{eq:contlyapunov} is readily satisfied
  when $0< \beta < \delta /q$.

   We end the proof by showing that~\eqref{eq:majoLW} holds. Similarly to
   the proof of Proposition~\ref{prop:difflyapunov}, we consider
   \[
   \mathscr{W}(x)=\e^{\theta |x|^q},  \quad \mathrm{with} \quad 0 <\theta < \beta,
   \]
   which satisfies the first condition in~\eqref{eq:majoLW}. Repeating the calculations
   leading to~\eqref{eq:boundLWf}, since~$\theta< \delta/q$ and
   $p< 2q - 2$, we obtain the existence of a constant~$c\geq 0$ such
   that
   \[
   \frac{\Lc \mathscr{W}}{\mathscr{W}}(x) + f(x) \leq - \theta q (\delta - \theta q)  |x|^{2q-2}
    + \theta q (q-1) |x|^{q-2} + a|x|^p \leq c,
   \]
   so the second condition in~\eqref{eq:majoLW} holds true,
   and Assumption~\ref{as:contlyapunov} is satisfied.
\end{proof}

%%%%%%%%%%%%%%%%%%%%%%%%%%%%%%%%%%%%%%%%%%%%%%%%%%%%%%%%%%%%%%%%%%%%%%%%%%%%
\subsection{Convergence results uniform with respect to the time step}
\label{sec:discretization}
When one considers continuous semigroups as in Section~\ref{sec:continuous}, it is natural
in practical applications to discretize~\eqref{eq:continuous} for example with
\begin{equation}
  \label{eq:simplediscr}
\Phi_k(\mu)(\varphi)=\frac{\E_{\mu} \left[ \varphi(x_k)\, \e^{\Dt \sum_{i=0}^{k-1} f(x_i)} \right]}
{\E_{\mu}\left[ \e^{\Dt \sum_{i=0}^{k-1} f(x_i)} \right]},
\end{equation}
where~$(x_k)_{k\in\N}$ is a discretization of the SDE~\eqref{eq:SDE} with time step $\Dt>0$,
\textit{i.e.}~$x_k$ is an approximation of~$X_{k\Dt}$.
First, as mentioned in~\cite{ferre2017error}, the stability of the discretization
schemes for unbounded state spaces was an open question. Our framework covers this situation,
as shown by the examples provided in Section~\ref{sec:discrappli}.

Another interesting consequence of our analysis is that we are able to obtain convergence
estimates uniform in the time step~$\Dt$, in the sense that the rate of decay towards the
invariant measure in fact depends on~$k\Dt$, the physical time of the system, with a prefactor
independent of~$\Dt$. It has been the purpose of several works to develop
such uniform in~$\Dt$ estimates for long time convergence, in
particular in the context of Metropolized discretizations of overdamped
Langevin dynamics~\cite{bou2012nonasymptotic,fathi2015improving},
discretization of the Langevin dynamics~\cite{lelievre2016partial,leimkuhler2016computation},
and other discretizations of SDEs~\cite{debussche2012weak,kopec2014weak,kopec2015weak}.
Our goal is to show that similar results can be obtained for Feynman--Kac semigroups. For
the remainder of this section,  we assume that
\[
\X=\mathbb{T}^d
\]
is the $d$-dimensional torus, the function~$\sigma$ in~\eqref{eq:SDE} is the identity matrix,
and we denote by~$\ceil{a}$ the upper integer part of~$a$ for $a\in\R$. Considering an unbounded
state space~$\X$ is also possible but, as noted in~\cite{ferre2017error}, this leads to serious
technical difficulties -- we therefore postpone this case to future works.

We consider here a simplified version of the framework extensively developed
in~\cite{ferre2017error}. We say that a kernel operator~$\Qdt^f$ defines a weakly consistent
discretization of the semigroup~\eqref{eq:continuous} if it satisfies
Assumption~\ref{as:regularity} and  there exist $\Dt^* >0$, $C>0$, $p\in\N$,
and an operator $\mathcal{R}_{\Dt}:C^{\infty}(\X)\to C^{\infty}(\X)$ (which encodes remainder terms)
such that, for any $\varphi\in C^{\infty}(\X)$,
\begin{equation}
  \label{eq:Qdtexpand1}
\Qdt^f \varphi = \varphi + \Dt (\Lc + f)\varphi + \Dt^2 \mathcal{R}_{\Dt}\varphi,
\end{equation}
where, for all $\Dt\in(0,\Dt^*]$,
\[
\| \mathcal{R}_{\Dt} \varphi \|_{\Linfty} \leq C\hspace{-0.15cm} \underset{\begin{array}{c}
    {\scriptstyle m \in \N^d}
\\{\scriptstyle | m |\leq p} \end{array} }{\sup} \| \partial^{m} \varphi \|_{\Linfty}, 
\]
 using the notation $\partial^m= \partial_{x_1}^{m_1}\hdots\partial_{x_d}^{m_d}$
for $ m = (m_1,\hdots,m_d)\in\N^d$.
The dynamics~\eqref{eq:continuous} is then approximated by the discrete semigroup
\begin{equation}
  \label{eq:discretization}
  \forall\, k \geq 1, \quad \forall\, \mu\in\PX,\quad \forall\, \varphi\in \Linfty(\X), \quad
  \Phi_k(\mu)(\varphi)= \frac{\mu\left( (\Qdt^f)^k \varphi\right)}{\mu\left( (\Qdt^f)^k
    \ind \right)}.
\end{equation}  

The latter definition encompasses many numerical schemes -- we refer the interested reader
to~\cite{ferre2017error} for a justification of this framework and the subsequent numerical
analysis. Although the framework is rather abstract, one can think for concreteness to the Euler--Maruyama scheme
for discretizing~\eqref{eq:SDE}, in which case the evolution operation~$\Qdt$ is defined as
\[
\forall\,x\in\X, \quad (\Qdt\varphi)(x) = \E[ \varphi(x_{n+1})\,|\, x_n=x],
\]
for the numerical scheme
\[
x_{n+1} = x_n + b(x_n)\Dt + \sqrt{\Dt}G_n,
\]
where~$(G_n)_{n\geq 1}$ is a familly of independent standard Gaussian variables.
It is then possible to check that, for instance, the left point integration
$\Qdt^f = \e^{\Dt f}\Qdt$ satisfies the expansion~\eqref{eq:Qdtexpand1}. The specific numerical scheme to be considered
is however not crucial at all, as long as it is a weak approximation of the underlying SDE (as made precise in~\eqref{eq:Qdtexpand1}) with a regularizing evolution operator~$\Qdt$.

In order to obtain uniform in the time step estimates,
we now assume a uniform minorization and boundedness condition of the following form.

\begin{assumption}
  \label{as:unif}
  Fix a time $T>0$. There exist
  $\Dt^* >0$, $\eta\in\PX$ and $\alpha\in(0,1)$ such that,
  for any $\Dt \in(0, \Dt^*]$, the operator~$\Qdt^f$ is strong Feller and for
  any $\varphi\in\Linfty(\X)$ with $\varphi \geq 0$, it holds
      \begin{equation}
        \label{eq:minounif}
  \forall\,x\in\X, \quad  \alpha \eta(\varphi)  \leq
  \left(\left(\Qdt^f\right)^{\ceil{\frac{T}{\Dt}}}\varphi\right) (x)
  \leq\frac{1}{\alpha} \eta(\varphi).
      \end{equation}
\end{assumption}

The lower bound in~\eqref{eq:minounif} corresponds to a minorization condition
with respect to a physical time~$T>0$, see~\cite[Section~3]{lelievre2016partial}.
The upper bound is a standard ingredient for
studying Feynman--Kac semigroups, see for instance~\cite{del2001stability,del2004feynman}.
We will see in Proposition~\ref{prop:unif} that Assumption~\ref{as:unif} is naturally satisfied
if a similar condition holds for~$\Qdt$ and the evolution operator reads
$\Qdt^f = \e^{\Dt f}\Qdt$ (which corresponds to
the discretization~\eqref{eq:simplediscr}).

\begin{remark}
  Although Assumption~\ref{as:unif} holds in many situations when~$\X$ is compact, the requirement
  that the upper bound in~\eqref{eq:minounif} holds may not seem natural in view of
  the results of Section~\ref{sec:discrete}. Indeed, our framework
  shows that this upper bound is not necessary to prove the ergodicity of Feynman--Kac semigroups,
  as opposed to previous
  works~\cite{del2001stability,del2002stability,del2004feynman,ferre2017error}. A careful look
  at the proof of Theorem~\ref{theo:unif} shows that this upper bound is only used to show
  the uniform boundedness of the approximate eigenvector~$h_{\Dt}$ in~\eqref{eq:hdt}. However,
  controlling~$h_{\Dt}$ as $\Dt\to 0$ does not seem to be an easy task without the upper bound
  in~\eqref{eq:minounif}. We therefore stick  to this assumption here.
\end{remark}

Before stating our uniform in~$\Dt$ convergence result, we need the following estimate
deduced from~\cite[Lemma~5]{ferre2017error}, whose proof can be found in
Appendix~\ref{sec:estimateh}.

\begin{lemma}
  \label{lem:estimateh}
  Consider the process~$(X_t)_{t\geq 0}$ solution to~\eqref{eq:SDE} with
  $\sigma = \id$, $b\in C^{\infty}(\X)$, and a function $f\in C^{\infty}(\X)$.
  Then the operator $\Lc + f$ admits a real isolated largest (in modulus) eigenvalue~$\lambda$ 
  with eigenvector $h\in C^{\infty}(\X)$ and associated eigenspace of dimension one, which satisfies
  \[
  (\Lc + f) h = \lambda h, \quad \mbox{and}\quad P_t^f h = \e^{t\lambda} h,\quad \forall\, t \geq 0.
  \]
  If~$\Qdt^f$ corresponds to a weakly consistent discretization of~\eqref{eq:continuous} (\textit{i.e.}~\eqref{eq:Qdtexpand1}
  holds) satisfying Assumption~\ref{as:unif}, then for any $\Dt >0$, the operator~$\Qdt^f$ has a
  largest (in modulus) eigenvalue $\Lambda_{\Dt}\in \R$, which is non-degenerate. Denoting the associated
  eigenvector by~$h_{\Dt}$, 
  \[
  \Qdt^f h_{\Dt} = \Lambda_{\Dt} h_{\Dt},
  \]
  with the normalization $\eta(h_{\Dt})=1$, there exist $\Dt^* >0$, $C>0$,
  $\varepsilon>0$  such that for all $\Dt\in(0,\Dt^*]$, there is $c_{\Dt}\in\R$ for which
  \begin{equation}
    \label{eq:lambdadt}
    \Lambda_{\Dt} = \e^{\Dt \lambda + \Dt^2 c_{\Dt}},
  \end{equation}
  with $| c_{\Dt} | \leq C$ and
  \begin{equation}
    \label{eq:hdt}
    \forall\, x\in\X,\quad \forall\, \Dt\in(0,\Dt^*],\quad
    \varepsilon\leq   h_{\Dt} (x) \leq \varepsilon^{-1}.
  \end{equation}
\end{lemma}
Lemma~\ref{lem:estimateh} means that the evolution operator associated with a weakly consistent
discretization has a principal eigenvalue approximating the principal eigenvalue
of the continuous dynamics, and that its associated principal eigenvector remains uniformly
bounded from below and above if~$\Dt$ is sufficiently small.
 Let us now
state the uniform in~$\Dt$ version of Theorem~\ref{theo:general}.

\begin{theorem}
\label{theo:unif}
Consider a consistent discretization~$\Qdt^f$ of the dynamics~\eqref{eq:continuous}
satisfying Assumption~\ref{as:unif}. Then, there exists~$\Dt^*>0$ such that,
for any $\Dt \in(0,\Dt^*]$, the dynamics~\eqref{eq:discretization} admits a unique invariant measure
$\mu_{f,\Dt}^\star\in\PX$. Moreover, there exist $\kappa >0$, $C>0$ such that for any
$\varphi \in \Linfty(\X)$, $\mu\in\PX$, and $\Dt\in(0,\Dt^*]$, it holds
  \[
  \forall\, k \geq 0, \quad
  \left| \Phi_k(\mu)(\varphi) - \mu_{f,\Dt}^\star(\varphi) \right| \leq
  C\, \e^{-\kappa k\Dt } \| \varphi\|_{\Linfty}.
\]
\end{theorem}
Let us note that the uniformity of the prefactor~$C$ in the initial condition is a consequence
of the boundedness of~$\X$. Indeed, in this case, we can choose $W\equiv 1$ as a Lyapunov
function, so the constant~$C_{\mu}$ in~\eqref{eq:Cmulast} can be uniformly
bounded using~\eqref{eq:hdt}. Such a uniformity does not hold for Theorem~\ref{theo:general}
since in that case~$\X$ was not assumed to be bounded. The important part of the theorem is
the control of~$C$ and~$\kappa$ with respect to the time step,
which provides convergence with respect to the physical time~$k\Dt$.

\begin{proof}
  The proof essentially relies on the fact that if~$\Qdt^f$ satisfies Assumption~\ref{as:unif},
  then~$\Qhdt$ defined as in Lemma~\ref{lem:Qh} satisfies a uniform minorization
  condition. For controlling the dependencies in the time step, we rely on
  Lemma~\ref{lem:estimateh}, and use the same notation.

  We want to prove a uniform minorization condition (in the sense
  of~\cite[Lemma 3.4]{lelievre2016partial}) for the operator defined by
  \[
  \Qhdt= \Lambda_{\Dt}^{-1} \hdt^{-1} \Qdt^f \hdt,
  \]
  and apply~\cite[Corollary 3.5]{lelievre2016partial}. Fix $T>0$. From~\eqref{eq:minounif}
  and~\eqref{eq:hdt} we have, for any $\varphi\geq 0$ and $x\in\X$,
  \begin{equation}
    \label{eq:intQh}
    \Qhdt^{\ceil{\frac{T}{\Dt}}}\varphi(x) = \Lambda_{\Dt}^{-\ceil{\frac{T}{\Dt}}} \hdt^{-1}
    \left(\Qdt^f\right)^{\ceil{\frac{T}{\Dt}}} (\hdt\varphi)(x)
    \geq  \Lambda_{\Dt}^{-\ceil{\frac{T}{\Dt}}} \varepsilon^2 \alpha \eta(\varphi).
  \end{equation}
  Moreover, from~\eqref{eq:lambdadt},
  \[
  \Lambda_{\Dt}^{-\ceil{\frac{T}{\Dt}}} =
  \e^{- \Dt (\lambda + \Dt c_{\Dt})  \ceil{\frac{T}{\Dt} } }
  \geq \e^{-2|\lambda| T} >0,
  \]
  upon possibly reducing~$\Dt^*$. Then,~\eqref{eq:intQh} becomes
  \[
  \forall\, x \in \X, \quad
  \Qhdt^{\ceil{\frac{T}{\Dt}}}(x,\cdot) \geq  \alpha\varepsilon^2\,  \e^{-2|\lambda| T} \eta(\cdot).
  \]
  As a result,~$\Qhdt$ satisfies the assumptions of~\cite[Corollary 3.5]{lelievre2016partial}:
  there exist a unique measure $\mu_{h,\Dt}\in\PX$, $C>0$, $\kappa >0$ such that, for any
  $\phi\in\Linfty(\X)$, $k\in\N$ and $\Dt \in(0, \Dt^*]$,
  \[
  \left\| \Qhdt^k\phi - \mu_{h,\Dt}(\phi) \right\|_{\Linfty} \leq C\, \e^{-\kappa k \Dt}
  \| \phi \|_{\Linfty}.
  \]
  This is a version of Lemma~\ref{lem:Qh}  uniform with respect to~$\Dt$.
  The result then follows by rewriting the proof of Theorem~\ref{theo:general},
  with~$\bar{\alpha}^k$ replaced by~$\e^{-\kappa k \Dt}$.

  It only remains to study the constant~$C_{\mu,\Dt}$ arising in Theorem~\ref{theo:general}
  (see~\eqref{eq:Cmulast}), which now also depends on~$\Dt$ through the eigenvector~$h_{\Dt}$
  and the invariant measure~$\mu_{h,\Dt}$.
  Since~$\X$ is bounded, we can actually choose a constant Lyapunov function, \textit{i.e.}
  $W=\ind$. Next, using~\eqref{eq:hdt} we obtain that for any~$\Dt\in(0,\Dt^*]$ and
  any~$\mu\in\PX$, it holds
  \[
  C_{\mu,\Dt} =  \frac{4}{\mu_{h,{\Dt}}(h_{\Dt}^{-1})}
  \big( 1 + \mu_{h,\Dt}(h_{\Dt}^{-1})\big) \frac{1}{\mu(h_{\Dt})}
  \leq 4 \varepsilon^{-2}(1 + \varepsilon^{-1}).
  \]
  This provides a uniform bound on~$C_{\mu,\Dt}$, which concludes the proof.
\end{proof}

We now show that the setting of Theorem~\ref{theo:unif} is natural, since
Assumption~\ref{as:unif} can be deduced from a similar assumption on the Markov dynamics $\Qdt$
when the evolution operator is $\Qdt^f=\e^{\Dt f} \Qdt$, which
corresponds to the discretization~\eqref{eq:simplediscr}. For
proving the condition on $\Qdt$, we refer to~\cite{lelievre2016partial} and the references
therein.
\begin{prop}
  \label{prop:unif}
  Assume that~$\X$ is bounded, $f\in C^{0}(\X)$, and the SDE~\eqref{eq:SDE}
  is discretized for a given time step $\Dt>0$ with a Markov chain~$(x_k)_{k\in\N}$
  whose evolution operator~$\Qdt$ is strong Feller and satisfies
  the following uniform minorization and boundedness condition: for a fixed $T>0$, there exist
  $\Dt^* >0$, $\eta\in\PX$ and $\alpha\in(0,1)$ such that, for any $\Dt\in(0,\Dt^*]$
    and $\varphi \in \Linfty(\X)$ with $\varphi \geq 0$,
    \[
 \forall\, x\in\X, \quad
  \alpha \eta(\varphi) \leq 
  \big(\Qdt\big)^{\ceil{\frac{T}{\Dt}}} \varphi(x) \leq 
  \frac{1}{\alpha}\eta(\varphi).
  \]
  Then, the transition operator~$\Qdt^f$ defined as $\Qdt^f= \e^{\Dt f}\Qdt$
  satisfies Assumption~\ref{as:unif}.
\end{prop}

\begin{proof}
  Since~$\Qdt$ is strong Feller and~$f$ is continuous,~$\Qdt^f$ is strong Feller.
  Then, for any $k\in\N$ and $\varphi\in\Linfty(\X)$,
  \[
  (\Qdt^f)^k\varphi(x) = \E_x\left[ \varphi(x_k)\, \e^{\Dt \sum_{i=0}^{k-1} f(x_i)}\right]
  \geq \e^{ -k \Dt \| f \|_{\Linfty}} \E_x\left[ \varphi(x_k) \right]
  = \e^{-k\Dt \| f \|_{\Linfty}} \big( (\Qdt)^k \varphi\big) (x).
  \]
  Taking $k=\ceil{T/\Dt}$ with~$0< \Dt \leq \Dt^*$ then shows that
  \[
  (\Qdt^f)^{\ceil{\frac{T}{\Dt}}}\varphi(x) \geq \e^{-2 T\| f \|_{\Linfty}}
  \big( Q^{\ceil{\frac{T}{\Dt}}}\varphi\big)(x)
  \geq \e^{-2 T\| f \|_{\Linfty}} \alpha \eta(\varphi).
  \]
  A similar computation for the upper bound allows to conclude the proof.
\end{proof}

%%%%%%%%%%%%%%%%%%%%%%%%%%%%%%%%%%%%%%%%%%%%%%%%%%%%%%%%%%%%%%%%%%%%%%%%%%%%%%%%%%%%%%%%%%%

\section{Discussion}
\label{sec:discussion}

The ideas developped in this work concerning the ergodicity of Feynman--Kac
semigroups solve several problems for which, to the best of our knowledge,
no solution was available. They are closely related to previous works and we want to highlight
two important connections.

First, as we mentionned in the introduction, our framework can be considered as an
extension of ergodic theory for Markov chains~\cite{meyn2012markov}, when the evolution
operator of the dynamics does not conserve probability. For this reason, we tried to formulate
our assumptions in the flavour of~\cite{hairer2011yet}. However, the spectral theory on which
we crucially rely in our study requires stronger conditions. 
This leaves open a few questions, as the converge of Feynman--Kac dynamics based on
Metropolis type kernels, which lack regularity,
or the case of non-Polish spaces, which may arise for stochastic partial differential equations.
Another interesting feature of our framework is that we
can prove ergodicity for Feynman--Kac dynamics for which the underlying Markov chain is not ergodic
-- a case we called Diffusion Monte Carlo (DMC) in analogy with quantum physics models
(see Proposition~\ref{prop:DMC}). Finally, we mention that it would be interesting to extend our spectral approach to situations 
where the expected rate of convergence is sub-geometric, see for instance~\cite{douc2004sub,douc2009sub} in the context of Markov chains. %For such more complex situations, a more appropriate approach might be to consider stochastic control arguments, see~\cite{chetrite2015var,nickelsen2018an}.

The other clear connection concerns Large Deviations theory. Indeed, one motivation
for studying Feynman--Kac dynamics is to prove large deviations principles for additive
functionals of Markov
chains~\cite{donsker1975variational,dembo2010large,wu2001large,kontoyiannis2005large},
which can be achieved by proving the existence of formulas such as~\eqref{eq:SCGF}. It is then
no surprise that the spectral theory we develop, although based on~\cite{bellet2006ergodic}, is
reminiscent of~\cite{kontoyiannis2005large}, and requires stronger assumptions than the
ones needed for proving ergodicity in~\cite{hairer2011yet}. However, the tools we use seem new
in this context, and more adapted to the situation at hand, for instance the Krein--Rutman
theorem based on the minorization condition. In particular,~\cite{kontoyiannis2005large}
(like~\cite{feng2006large}) makes use of nonlinear generators
related to an optimal control problem, and we show in~\cite{ferre2019fine} that our
linear spectral strategy can indeed be used for obtaining large deviations results in
weighted topologies, for possibly degenerate diffusions.

\subsection*{Acknowledgements}
The authors are grateful to Jonathan C. Mattingly for interesting discussions at a
preliminary stage of this work. The authors also warmly thank Nicolas Champagnat and Denis
Villemonais for pointing out a gap in one argument in the first version of the manuscript.
The PhD of Grégoire Ferré is supported by the Labex Bézout.
The work of Gabriel Stoltz was funded in part by the Agence Nationale de la Recherche, under
grant ANR-14-CE23-0012 (COSMOS). Gabriel Stoltz and Mathias Rousset are supported
by the European Research Council under the European
Union’s Seventh Framework Programme (FP/2007-2013)/ERC Grant Agreement number 614492.
The work of Mathias Rousset is supported by INRIA Rennes and IRMAR.
We also benefited from the scientific environment of the Laboratoire International
Associé between the Centre National de la Recherche Scientifique and the University of Illinois at
Urbana-Champaign.

%%%%%%%%%%%%%%%%%%%%%%%%%%%%%%%%%%%%%%%%%%%%%%%%%%%%%%%%%%%%%%%%%%%%%%%%%%%%%%%%%%%%%%%%%%%
\appendix

\section{Stability of Markov chains}
\label{sec:tools}

In this section, we recall the results presented in~\cite{hairer2011yet}. We consider a
measurable space~$\X$ and Markov chain~$(x_k)_{k\geq 0}$ with transition kernel~$Q$ on~$\X$.
By transition kernel, we mean that (i) for all $x\in\X$, $Q(x, \cdot)$ is a positive measure
on~$\X$, (ii) for any measurable set $A\subset\X$, $Q(\cdot,A)$ is measurable, and (iii) $Q\ind=\ind$.
In the notation of Section~\ref{sec:results}, $Q$ is a kernel operator (\textit{i.e.}~(i)
and~(ii) are satisfied) such that $Q\ind=\ind$.

The stability of Markov dynamics can be obtained from minorization and Lyapunov 
conditions~\cite{meyn2012markov,bellet2006ergodic,hairer2011yet}.

\begin{assumption}
  \label{as:lyapunov}
  There exist a function $\mathcal{W}:\X\to[1, +\infty)$ and constants $C \geq 0$, $\gamma \in (0,1)$
    such that
    \begin{equation}
      \label{eq:lyapunov}
      \forall \, x\in \X, \quad (Q\mathcal{W})(x) \leq \gamma \mathcal{W}(x) + C.
    \end{equation}  
\end{assumption}

Given such a Lyapunov function, we consider the associated functional space as
in~\eqref{eq:LW}.\begin{footnote}{Compared to~\cite{hairer2011yet}, we replace~$\mathcal{W}$
    by $\mathcal{W}+1$; this is for notational convenience only.}\end{footnote}
A second key ingredient in the ergodicity of $Q$ is the following minorization condition.

\begin{assumption}
  \label{as:minorization}
  There exist $\alpha\in (0,1)$ and $\eta\in\PX$ such that
  \begin{equation}
    \label{eq:minorization}
    \underset{x\in\mathcal{C}}{\inf}\ Q(x, \cdot) \geq \alpha \eta(\, \cdot\, ),
  \end{equation}
  where $\mathcal{C}=\{ x\in\X \, | \, \mathcal{W}(x) \leq R +1 \}$ for some $R> 2 C/(1-\gamma)$,
  and $\gamma$, $C$ are the constants from Assumption~\ref{as:lyapunov}.
\end{assumption}

The following result holds under these conditions (see~\cite[Theorem 1.2]{hairer2011yet}).

\begin{theorem}
\label{theo:HM}
Let Assumptions~\ref{as:lyapunov} and~\ref{as:minorization} hold. Then,~$Q$ has a unique
invariant measure~$\mu^\star$, which is such that $\mu^\star(\mathcal{W})<+\infty$. Moreover,
there exist $C >0$ and $\bar{\alpha}\in (0,1)$ such that, for any $\varphi\in \Linfty_{\mathcal{W}}(\X)$,
\[
\forall\, k\geq 0, \quad \| Q^k \varphi - \mu^\star(\varphi) \|_{\Linfty_{\mathcal{W}}}
\leq C \bar{\alpha}^k \| \varphi - \mu^\star(\varphi) \|_{\Linfty_{\mathcal{W}}}.
\]
\end{theorem}

\section{Useful theorems}
\label{sec:krein}
We remind here some definitions and results around the Krein--Rutman theorem, as well as some basic
results from analysis. Let us start with some operator theoretic definitions
from~\cite{nussbaum1970radius,reed1980functional,deimling2010nonlinear,bellet2006ergodic}.

\begin{definition}
  \label{def:operators}

  For a Banach space~$E$ and an operator $T\in\mathcal{B}(E)$, we denote by~$\Lambda(T)$
  its spectral radius defined by:
  \[
  \Lambda(T) = \underset{k\to+\infty}{\lim}\ \big\| T^k \big\|_{\mathcal{B}(E)}^{\frac{1}{k}}
  = \underset{k \geq 1}{\inf}\ \big\| T^k \big\|_{\mathcal{B}(E)}^{\frac{1}{k}}.
  \]
  We denote by~$\theta(T)$ the essential spectral radius of~$T$ defined by
  (see~\cite[Eq.~(1.14)]{nussbaum1998eigen} and~\cite[Theorem~1]{nussbaum1970radius}):
  \[
  \theta(T) = \underset{k\to+\infty}{\lim}\, \Big(\inf \big\{ \big\| T^k - Q
  \big\|_{\mathcal{B}(E)},\ Q
  \ \mathrm{ compact} \big\}\Big)^{\frac{1}{k}}
  = \underset{k \geq 1}{\inf}\, \Big(\inf \big\{ \big\| T^k - Q \big\|_{\mathcal{B}(E)},\ Q
  \ \mathrm{ compact} \big\}\Big)^{\frac{1}{k}}.
  \]

  An operator $T\in\mathcal{B}(E)$ is said to be compact if
  it maps bounded sets into precompact sets. In other words,~$T$ is compact if, for any
  bounded sequence~$(u_n)_{n\in\N}$ in~$E$, there is a subsequence~$(n_k)_{k\in\N}$
  such that~$(Tu_{n_k})_{k\in\N}$ converges in~$E$, see~\cite{reed1980functional}.

\end{definition}

In order to recall the Krein--Rutman theorem, let us first give some definitions for cones in
Banach spaces.
\begin{definition}
\label{def:cone}
Let~$E$ be a Banach space. A closed convex set $\Kb\subset E$ is said to be a cone if
$\Kb \cap -\Kb =\{0\}$ and for all $u\in \Kb$ and $\alpha \in \R_+$, it holds $\alpha u \in \Kb$.
A cone is total if the norm closure of $\Kb - \Kb$ is equal to~$E$.
\end{definition}

We now recall a weak version of the Krein--Rutman theorem, which can be found
in~\cite[Theorem~1.1]{nussbaum1998eigen}. Interesting remarks and comments are
also available in~\cite[Section~19.8]{deimling2010nonlinear}.

\begin{theorem}
\label{theo:weakkrein}
Let~$E$ be a Banach space, $\Kb\subset E$ a total cone, and $T\in\mathcal{B}(E)$ be such that
$\theta(T)<\Lambda(T)$ and $T\Kb\subset \Kb$. Then~$\Lambda(T)$ is an eigenvalue
of~$T$ with an eigenvector in~$\Kb$.
\end{theorem}

In Theorem~\ref{theo:weakkrein}, there is no uniqueness of the eigenvector. The non
degeneracy can be otained under stronger positivity conditions on the operator~$T$, as made precise
in~\cite[Theorems~19.3 and~19.5]{deimling2010nonlinear}.
In order to control the essential spectral radius and apply the Krein--Rutman theorem, we will need
the following classical results, see~\cite[Theorem~11.28]{rudin1991real}
and~\cite[Theorem~2.7.19]{schwartz1991analyse}.

\begin{theorem}[Ascoli]
  \label{theo:ascoli}
  Let $(\mathcal{Y},\mathrm{d}_{\mathcal{Y}})$ be a compact metric space 
  and~$C^0(\mathcal{Y})$ be the space of continuous functions over~$\mathcal{Y}$ endowed with the
  uniform norm $\|f\|_{C^0}=\sup_{y\in\mathcal{Y}}|f(y)|$. Consider a uniformly
  bounded and equicontinuous sequence~$(f_n)_{n\in\N}$, \textit{i.e.} a sequence for
  which there exists $M>0$ such that $\| f_n \|_{C^0}\leq M$ for all $n\geq 1$, and
  for any $\varepsilon>0$ there exists $\delta >0$ such that
  $\mathrm{d}_{\mathcal{Y}}(x,y)\leq \delta$ implies $|f(x) - f(y)|\leq \varepsilon$.
  Then~$(f_n)_{n\in\N}$ converges in the uniform norm to some limit~$f$ up to extraction.
\end{theorem}

\begin{theorem}[Heine--Cantor]
  \label{theo:heine}
  Consider $f:E\to F$ where $(E,\mathrm{d}_E)$ and $(F,\mathrm{d}_F)$ are two metric spaces and~$E$
  is compact.
  Then, if~$f$ is continuous, it is uniformly continuous: for any $\varepsilon >0$, there is
  $\delta >0$ such that for any $x$, $x'\in E$ with $\mathrm{d}_E( x, x') \leq \delta$,
  it holds $\mathrm{d}_F( f(x) , f(x') ) \leq \varepsilon$.
\end{theorem}

We close this section with some results in probability theory.
The next lemma can be found in~\cite[Lemma~4.14]{hairer2006ergodic}.

\begin{lemma}
  \label{lem:tight}
  If $\X$ is a Polish space and $\mu\in\PX$, then the familly constituted of the single
  measure~$\mu$ is tight, \textit{i.e.} for any $\varepsilon>0$, there exists a compact set
  $K\subset\X$ such that $\mu(K)\geq 1 - \varepsilon$.
\end{lemma}
We finally present results concerning ultra-Feller operators, extending the
ones of~\cite[Appendix A]{hairer2009non} for transition kernels that are not normalized.
Recall that the total variation distance between
two positive measures $\mu,\nu\in\mathcal{M}(\X)$ is defined by:
\begin{equation}
  \label{eq:TVdist}
\| \mu - \nu\|_{\mathrm{TV}} =
\underset{
  \begin{array}{c}{\scriptstyle \varphi\in\Linfty(\X)}
    \\ {\scriptstyle\|\varphi\|_{\Linfty}\leq 1}
  \end{array}
} {\sup}
\ \intX \varphi \, d\mu - \intX \varphi \, d\nu.
\end{equation}
\begin{definition}[Ultra-Feller]
  \label{def:ultra}
  A kernel operator~$Q$ is ultra-Feller if the mappping
  $x\mapsto Q(x,\cdot) \in\mathcal{M}(\X) $ is continuous in the total variation
  distance~\eqref{eq:TVdist}.
  \end{definition}

The next lemma, used to show that an operator is ultra-Feller, is adapted
from~\cite[Appendix A]{hairer2009non}.

\begin{lemma}
  \label{lem:PQ}
  Suppose that $P$ and $Q$ are two  kernel operators over a Polish space~$\X$
  that satisfy the following properties:
  \begin{itemize}
  \item for all~$\varphi\in\Linfty(\X)$, $Q\varphi$ is continuous and finite;
  \item for all~$\psi$ such that~$|\psi| \leq Q\ind$, $P\psi$ is continuous and finite.
  \end{itemize}
  Then $P Q$ is ultra-Feller.
\end{lemma}

We remind some elements of the proof from~\cite[Theorem~1.6.6]{hairer2009non}, which is based
on the Banach--Alaoglu theorem. The details are left to the reader.
\begin{proof}
  A first element to prove Lemma~\ref{lem:PQ} is to show that, if~$Q$ is strong Feller, then
  there exists a reference probability measure~$\zeta\in\PX$ such that for any $x\in\X$, $Q(x, \cdot)$ is
  absolutely continuous with respect to~$\zeta$. This is shown in~\cite[Lemma~1.6.4]{hairer2009non}
  for operators~$Q$ such that $Q\ind=\ind$. Even for a non-probabilistic~$Q$, we can consider
  the normalized probabilities
  \[
  \frac{Q(x,\cdot)}{Q\ind(x)},
  \]
  for~$x$ in the open set $\widetilde{X}:=\{x\in\X\,|\, Q\ind(x)>0\}$.
  We can apply~\cite[Lemma~1.6.4]{hairer2009non} to these probabilities defined
  over the set~$\widetilde{X}$, so there exists a
  measure~$\zeta$ such that, for any~$x\in\widetilde{X}$, $Q(x,\cdot)$ is absolutely continuous
  with respect to~$\zeta$. If $x\in\X\setminus\widetilde{X}$, $Q(x,\cdot) = 0$, which is
  also absolutely continuous with respect to~$\zeta$, so that~$Q(x,\cdot)$ is absolutely
  continuous with respect to~$\zeta$ for any~$x\in\X$.

  Once this is done, one can write the kernel~$Q$ as~$Q(y,dz)=k(y,z)\zeta(dz)$ with
  $k(y,\cdot)\in L^1(\X,\zeta)$ for all~$x\in\X$.
  If one supposes by contradiction that $PQ$ is not ultra-Feller, then Definition~\ref{def:ultra}
  shows that there exist a sequence of functions~$(g_n)_{n\in\N}$ with~$\|g_n\|_{\Linfty}\leq 1$
  and a sequence~$(x_n)_{n\in\N}$ converging to an element~$x\in\X$ such that for
  some~$\delta >0$ it holds
  \begin{equation}
    \label{eq:ultraPQ}
  \forall\, n \in\N, \quad PQ g_n(x_n) - PQg_n(x) > \delta.
  \end{equation}
  Since the sequence~$(g_n)_{n\in\N}$ is bounded,
  it possesses a weak-$\ast$ converging subsequence in~$L^{\infty}(\X,\zeta)$ (the space of
  $\zeta$-essentially bounded functions) to an element~$g\in\Linfty(\X,\zeta)$. In
  particular it holds (upon extracting a subsequence), for any~$y\in\X$,
  \[
  \lim_{n\to +\infty} Q g_n(y) = \lim_{n\to +\infty} \intX k(y,z)g_n(z)\zeta(dz) = 
    \intX k(y,z)g(z)\zeta(dz) = Q g(y).
   \]
   Defining, $f_n = Q g_n$, the latter limit shows that~$f_n$ converges pointwise
   to~$f= Q g$. Since $(g_n)_{n\in\N}$ is bounded in~$\Linfty(\X)$, the second
   condition in Lemma~\ref{lem:PQ} ensures that $Pf_n(x)\to Pf(x)$ for all~$x\in\X$, by the
   dominated convergence theorem. This is the main difference compared to the proof
   in~\cite[Theorem~1.6.6]{hairer2009non}. The contradiction follows similarly.
   Indeed, defining the positive decreasing function~$h_n = \sup_{m\geq n} | f_m - f|$
   we have, for any~$m\in\N$,
   \[
   \lim_{n\to +\infty} Ph_n(x_n)\leq \lim_{n\to +\infty} Ph_m(x_n) =  Ph_m(x),
   \]
   so that $P h_n(x_n)\to 0$ as $n\to +\infty$. In the end,
   \[
   \lim_{n\to+\infty} P f_n(x_n) - Pf (x) \leq \lim_{n\to+\infty}
   |P f_n(x_n) - Pf (x_n)| + \lim_{n\to+\infty}
   |P f(x_n) - Pf (x)| = 0,
   \]
   which comes in contradiction with~\eqref{eq:ultraPQ} and concludes the proof.
\end{proof}

%%%%%%%%%%%%%%%%%%%%%%%%%%%%%%%%%%%%%%%%%%%%%%%%%%%%%%%%%%%%%%%%%%%%%%%%%%%%%%%%%
\section{Proof of Lemma~\ref{lem:preliminary}}
\label{sec:preliminary}
Let us show that $\mu(Q^f\ind)>0$ for any $\mu\in\PX$. First, Lemma~\ref{lem:tight} in
Appendix~\ref{sec:krein} ensures that, for any $\varepsilon >0$, there exists a
compact $K\subset\X$ such that $\mu(K)\geq 1 - \varepsilon$.
Consider next a compact set~$K_n$ of Assumption~\ref{as:lyapunovgeneral} such
that $K\subset K_n$. Then, with the corresponding $\alpha_n >0 $ and $\eta_n\in\PX$ defined in
Assumption~\ref{as:minogeneral}, we have
\[
\forall\, x\in K_n, \quad
(Q^f\ind)(x)\geq \alpha_n \eta_n(\ind)\geq \alpha_n >0.
\]
Integrating with respect to~$\mu$ leads to
\[
\intX (Q^f\ind)(x) \mu(dx) \geq \int_{K_n} (Q^f\ind)(x) \mu(dx) \geq \alpha_n \int_{K_n}\mu(dx)
= \alpha_n \mu(K_n) \geq \alpha_n(1-\varepsilon) >0,
\]
since $K\subset K_n$, which proves the statement.
Moreover $ W\geq 1$, so~\eqref{eq:lyapunovgeneral} implies that
$\mu(Q^f \ind)\leq \mu( Q^f W) <+\infty$ if $\mu(W)<+\infty$.

Since $W\geq 1$, we immediately have that $\eta_n(W) \geq 1 >0$ for any $n\geq 1$.
Now, for any $n\geq 1$ and $x\in K_n$, Assumptions~\ref{as:lyapunovgeneral}
and~\ref{as:minogeneral} lead to
\[
\alpha_{n}\eta_{n}(W) \leq Q^f W (x) \leq \gamma_n W(x) + b_n\ind_{K_n}(x) < +\infty,
\]
since~$W$ is finite. Moreover, $\alpha_{n}>0$, so that $\eta_n(W) < +\infty$ for any $n\geq 1$.

Let us conclude with the proof of~\eqref{eq:nbar}. We proceed by contradiction
and assume that, for any $n\geq n_0$, we have $\eta_n(K_n)=0$, with $n_0\geq 1$
an arbitrary integer. Consider $m\geq n_0$
and $\varphi_m=\ind_{K_m}\geq 0$. Then, using~\eqref{eq:minogeneral} with
$n=n_0$,
\[
\forall\, x\in K_{n_0},\quad \big( Q^f\varphi_m\big) (x) \geq \alpha_{n_0} \eta_{n_0}(K_m).
\]
Using again Lemma~\ref{lem:tight} in Appendix~\ref{sec:krein}, we see that for $m$
large enough, $\big( Q^f\varphi_m\big) (x) >0$ for $x\in K_{n_0}$ and so $Q^f\varphi_m\neq 0$.
However, for $n\geq m$, we have, using that $K_m\subset K_n$ (since the sets are
increasing):
\[
0\leq \eta_n(\varphi_m)=\eta_n(K_m)\leq \eta_n(K_n) = 0,
\]
since we assumed $\eta_n(K_n) = 0$ for $n\geq {n_0}$. The contradiction
with~\eqref{eq:irreducibility} shows that there exists $\nb\geq {n_0}$ such that
$\eta_{\nb}(K_{\nb}) >0$. Since~${n_0}$ is arbitrary, $\nb$ can be chosen arbitrarily large,
and this concludes the proof of Lemma~\ref{lem:preliminary}.

\section{Proof of Lemma~\ref{lem:specQV}}
\label{sec:specQV}

The proof is decomposed into three steps. First we show that the essential spectral radius of the
operator~$Q^f$ considered over~$\Linfty_W(\X)$ is zero. We next prove that the spectral radius~$\Lambda$
of~$Q^f$ is positive. Finally, we use the Krein-Rutman theorem to obtain that~$\Lambda$
is a eigenvalue of~$Q^f$ with largest modulus, and that the associated eigenvector is
positive.

\subsubsection*{Step 1: $Q^f$ has zero essential spectral radius}

We first perform the following decomposition, for any $n\geq 1$:
\[
(Q^f)^2 =  \ind_{K_n}Q^f\ind_{K_n}Q^f + \ind_{K_n^c}(Q^f)^2 + \ind_{K_n}Q^f\ind_{K_n^c}Q^f,
\]
where $K_n\subset\X$ are the compact sets from Section~\ref{sec:discrete}.
Applying again~$Q^f$ leads to 
\begin{equation}
  \label{eq:decompQV}
  (Q^f)^3 =  (\ind_{K_n}Q^f\ind_{K_n})^2 Q^f + \ind_{K_n^c}Q^f(\ind_{K_n}Q^f)^2 +
  Q^f\ind_{K_n^c}(Q^f)^2 + Q^f\ind_{K_n}Q^f\ind_{K_n^c}Q^f.
\end{equation}
We will show that $Q_n^f:=\ind_{K_n}Q^f\ind_{K_n}$ is such that~$(Q_n^f)^2$ is compact
on~$\Linfty_W(\X)$,
while~$\ind_{K_n^c}Q^f$ tends to zero in norm. This will prove that~$(Q^f)^3$ is compact as
limit of compact operators in operator norm, so the essential spectral radius of~$Q^f$ 
in~$\Linfty_W(\X)$, denoted by~$\theta(Q^f)$, is equal to zero.

Let us first prove that~$(Q_n^f)^2$ is compact on~$\Linfty_W(\X)$ for any $n\in\N$.
For this, we use the ultra-Feller property proved in Lemma~\ref{lem:PQ} (see
Appendix~\ref{sec:krein}) to apply the Ascoli theorem.
Consider a sequence~$(\varphi_k)_{k\in\N}$ in~$\Linfty_W(\X)$ such that $\| \varphi_k \|_{\Linfty_W} \leq M$
for some $M \geq 0$. By Assumption~\ref{as:regularity}, the operator $Q_n^f$ is strong Feller over
the compact set~$K_n$. In particular, for~$\varphi\in\Linfty_W(\X)$, $\varphi\ind_{K_n}\in\Linfty(\X)$,
so~$Q_n^f\varphi$ is continuous over~$K_n$ and finite, so that
Lemma~\ref{lem:PQ} in Appendix~\ref{sec:krein} applies. Indeed, the second condition in the lemma
is easy to check since~$Q_n\ind$ is equal to zero outside the compact~$K_n$.  Therefore,~$(Q_n^f)^2$ is
ultra-Feller by Lemma~\ref{lem:PQ}. By Definition~\ref{def:ultra}, the application
$x\in K_n \mapsto (Q_n^f)^2(x,\cdot)\in\mathcal{M}(\X)$ is continuous in total variation norm.
Since~$K_n$ is compact in the metric space~$\X$ and~$\PX$ is a metric space,
the Heine-Cantor theorem (Theorem~\ref{theo:heine} in Appendix~\ref{sec:krein}) ensures that
this application is continuous over~$K_n$. This means that, for any $\varepsilon >0$,
there exists $\delta >0$ such that, for any $x,x'\in K_n$ with $| x - x'|\leq \delta$, it holds
\begin{equation}
  \label{eq:intepsilon}
\underset{\| \varphi \|_{\Linfty}\leq 1}{\sup}\
\Big| \big( (Q_n^f)^2\varphi\big)(x) - \big( (Q_n^f)^2\varphi\big)(x') \Big| \leq \varepsilon. 
\end{equation}
Noting that Assumption~\ref{as:lyapunovgeneral} implies that
$1\leq \sup_{K_n} W < +\infty$, it holds $M_n = (\sup_{K_n}W)^{-1} \in (0,1]$
for any $n\geq 1$, so
  \begin{equation}
    \label{eq:Mn}
\big\{ \varphi \mbox{ measurable}\ \big| \ \| \ind_{K_n} \varphi \|_{\Linfty}\leq 1 \big\}
\supset \big\{ \varphi \mbox{ measurable}\ \big|\ \| \ind_{K_n} \varphi \|_{\Linfty_W}\leq M_n \big\}.
\end{equation}
Since $Q_n^f = \ind_{K_n}Q^f\ind_{K_n}$,~\eqref{eq:Mn} shows that~\eqref{eq:intepsilon}
implies
\[
\underset{\| \varphi \|_{\Linfty_W}\leq M_n}{\sup}\
\Big| \big( (Q_n^f)^2\varphi\big) (x) - \big( (Q_n^f)^2\varphi \big) (x') \Big| \leq \varepsilon. 
\]
As a consequence, if~$(\varphi_k)_{k\in\N}$ is such that $\| \varphi_k \|_{\Linfty_W} \leq M$,
we see that~$\big((Q_n^f)^2\varphi_k\big)_{k\in\N}$ is equicontinuous.
By the Ascoli theorem, it therefore converges uniformly to a continuous limit on~$K_n$
(since the function is supported on~$K_n$, we extend it by~$0$ on~$\X$ outside~$K_n$).
Since $W\geq 1$, it also converges as a function in~$\Linfty_W(\X)$, showing that~$(Q_n^f)^2$ is
a compact operator on~$\Linfty_W(\X)$. Since~$Q^f$ is bounded over~$\Linfty_W(\X)$ and the space
of compact operators is stable by composition with bounded
operators~\cite{reed1980functional},~$(Q_n^f)^2 Q^f$ is also compact.

We now show that the second, third and fourth operators on the right hand
side of~\eqref{eq:decompQV} tend to~$0$ in the operator norm of~$\Linfty_W(\X)$. For any
$\varphi\in\Linfty_W(\X)$,
\[
\big\| \ind_{K_n^c} Q^f\varphi \big\|_{\Linfty_W}  =
  \left\|\frac{  \mathds{1}_{K_n^c} Q^f\varphi}
             {W}   \right\|_{\Linfty}
              \leq \| \varphi \|_{\Linfty_W} \left\| \ind_{K_n^c} \frac{Q^f W}{W}
             \right\|_{\Linfty}
 \leq \gamma_n  \| \varphi \|_{\Linfty_W}.
\]
Taking the supremum over $\varphi\in\Linfty_W(\X)$ and using 
$\gamma_n\to 0$ as $n\to +\infty$, we obtain:
\begin{equation}
  \label{eq:Wcomp}
 \left\| \mathds{1}_{K_n^c} Q^f \right\|_{\B(\Linfty_W)} \xrightarrow[n\to +\infty]{} 0. 
 \end{equation}
Since~$Q^f$ is bounded on~$\B(\Linfty_W)$, the second, third and fourth operators on the right
hand side of~\eqref{eq:decompQV} vanish in norm as $n\to +\infty$. As a result,~$(Q^f)^3$ is the
 norm-limit of the compact operators~$(Q_n^f)^2 Q^f$ as $n\to+\infty$
 in~$\B(\Linfty_W)$.  Since the set of compact operators over~$\Linfty_W(\X)$ is closed in the
 Banach space~$\B(\Linfty_W)$,~$(Q^f)^3$ is compact, see
 \textit{e.g.}~\cite[Theorem~VI.12]{reed1980functional}. Using Definition~\ref{def:operators},
 we conclude that $\theta(Q^f)=0$. In this procedure, we see that working in the weighted
 space~$\Linfty_W(\X)$ as opposed to~$\Linfty(\X)$ is crucial in order to obtain the compactness
 of~$(Q^f)^3$ from the control~\eqref{eq:Wcomp} provided by the
 Lyapunov condition~\eqref{eq:lyapunovgeneral}.

 %%%%%%%%%%%%%%%%%%%%%%%%%%%%%
\subsubsection*{Step 2: The spectral radius is positive}
We now show that the spectral radius~$\Lambda$ of~$Q^f$ defined in~\eqref{eq:gelfand} is
positive, in order to use Theorem~\ref{theo:weakkrein}. Given the definition of the operator norm,
choosing some arbitrary non negative function $\phi\in\Linfty_W(\X)$ with
$\|\phi \|_{\Linfty_W}\leq 1$ leads to
\[
\big\| Q^f \big\|_{\mathcal{B}(\Linfty_W)} \geq \left\|
\frac{ Q^f\phi}{W}\right\|_{\Linfty} \geq \frac{ \left(Q^f\phi\right)(x_0)}{W(x_0)}, 
\]
where $x_0\in\X$ is arbitrary. We now consider a compact set corresponding to
some $n=\nb$ as defined in Lemma~\ref{lem:preliminary}, which satisfies $\eta_\nb(K_\nb) >0$, and
take $x_0\in K_\nb$. For any non negative function $\phi\in \Linfty_W(\X)$
with $\|\phi \|_{\Linfty_W}\leq 1$, 
\begin{equation}
  \label{eq:intmino}
  \eta_\nb\left(Q^f\phi\right)  = \left( \int_{K_\nb} (Q^f\phi)(x)\,\eta_\nb(dx)
  + \int_{\X\setminus K_\nb} (Q^f\phi)(x)\,\eta_\nb(dx) 
  \right)
  \geq  \int_{K_\nb} \alpha_\nb \eta_\nb(\phi)\, \eta_\nb(dx)
  \geq \alpha_\nb \eta_\nb(\phi)\, \eta_\nb(K_\nb),
\end{equation}
where we used~\eqref{eq:minogeneral} with $n=\nb$. Iterating the
inequality shows that 
\[
\forall\, k\geq 1, \quad
\eta_\nb\left( (Q^f)^k\phi \right) \geq \alpha_\nb^k \eta_\nb(K_\nb)^k \eta_\nb(\phi).
\]
This leads to the following lower bound on the operator norm of~$(Q^f)^k$:
\[
\begin{aligned}
  \big\| (Q^f)^k \big\|_{\mathcal{B}(\Linfty_W)}
  & \geq \frac{ \left((Q^f)^k\phi\right)(x_0)}{W(x_0)}
   = \frac{ \left(Q^f ((Q^f)^{k-1}\phi)\right)(x_0)}{W(x_0)}
  \\ & \geq \alpha_\nb \frac{ \eta_\nb\left((Q^f)^{k-1}\phi\right)}{W(x_0)}
  \geq \frac{\alpha_\nb^k \eta_\nb(K_\nb)^{k-1}}
       {W(x_0)}\eta_{\nb}(\phi).
\end{aligned}
\]
Taking the power~$1/k$ and the limit $k\to +\infty$, together with the choice
$\phi=\ind\in\Linfty_W(\X)$, leads to
\[
\Lambda \geq \alpha_\nb \eta_\nb(K_\nb).
\]
From Lemma~\ref{lem:preliminary}, it holds $\eta_\nb(K_\nb) >0$,
hence $\Lambda >0$ and~$Q^f$ has a positive spectral radius. Note that the
existence of $\nb\geq 1$ such that $\eta_\nb(K_\nb) >0$ is crucial for this step.

%%%%%%%%%
\subsubsection*{Step 3: Existence of a principal eigenvector}

In order to use Theorem~\ref{theo:weakkrein}, we introduce the closed cone:
\[
\Kb_W=\big\{u \in \Linfty_W(\X) \ \big| \ u \geq 0 \big\}.
\]
This cone is total, and the positivity of $Q^f\in\Linfty_W(\X)$ shows that
$Q^f \Kb \subset \Kb$. At this stage, Theorem~\ref{theo:weakkrein} in
Appendix~\ref{sec:krein} ensures that the spectral radius~$\Lambda$ is an
eigenvalue of~$Q^f$ of largest modulus with an associated eigenvector
$h\in\Kb_W\setminus \{0\}$.

%%%%%%%%%%%%%%%%%%%%%%%
\subsubsection*{Step 4: Positivity}
We now use the irreducibility condition~\eqref{eq:irreducibility} to
show that, for the eigenvector~$h$ obtained in Step~3, it holds
$h(x)>0$ for all $x\in\X$ and hence $\eta_n(h)>0$ all~$n\geq 1$.

Let us show the first property by contradiction. Assume that there exists
$x_0\in\X$ such that $h(x_0)=0$. Since the sets~$K_n$ are increasing, there
exists~$n_0$ such that for all $n\geq n_0$ it holds $x_0\in K_n$ so that,
by~\eqref{eq:minogeneral},
\[
\forall\, n\geq n_0,\quad \big(Q^f h\big) (x_0)\geq \alpha_n \eta_n(h).
\]
Since $ Q^f h = \Lambda h$ with $\Lambda >0$, this leads to
\[
0 \geq \eta_n(h),
\]
and so $\eta_n(h) = 0$ for $n\geq n_0$. By the irreducibility assumption~\eqref{eq:irreducibility},
we therefore have $(Q^f h)(x)=0$ for all $x\in\X$. Using again $Q^fh = \Lambda h$,
this shows that $h=0$, which is in contradiction with the fact that~$h$ is an
eigenvector associated with~$\Lambda$.

The second property follows from $h(x) >0$ for all~$x\in\X$ and $\eta_n\in\PX$ for all $n\geq 1$.
Indeed, 
\begin{equation}
  \X = \bigcup_{k\geq 1}\ h^{-1}\Big[ \frac{1}{k},+\infty\Big),
\end{equation}
where~$h^{-1}$ denotes here the pre-image of~$h$. Therefore, for a given $n\geq 1$,
\[
\eta_n(\X) = \eta_n\big( h^{-1}[1,+\infty)\big)
  + \sum_{k\geq 1} \eta_n \Big( h^{-1}\Big[\frac{1}{k+1},\frac{1}{k}\Big)\Big) = 1.
\]
Thus, there exists $N\geq 1$ such that
\[
\eta_n\Big( h^{-1}\Big[ \frac{1}{N},+\infty\Big) \Big) \geq \frac{1}{2},
\]
so
\[
\eta_n(h) \geq \eta_n \big( h\ind_{ h\geq \frac{1}{N} }\big)
\geq \frac{1}{N}\eta_n \big(\ind_{ h\geq \frac{1}{N} }\big)\geq \frac{1}{2N}.
\]
Since $n\geq 1$ is arbitrary, this shows that $\eta_n(h)>0$ for all $n\geq 1$.

%%%%%%%%%%%%%%%%%%%%%%%%%%%%%%%%%%%%%%%%%%%%%%%%%%%%%%%%%%%%%%%%%%%%%%%%%%
\section{Proof of Lemma~\ref{lem:Qh}}
\label{app:Qh}
A first important remark is that~$Q_h$ is a Markov operator. Indeed, it is a well-defined
kernel operator (since $0<h(x)<+\infty$ for all $x\in\X$),
and $Q_h\ind = \Lambda^{-1}h^{-1} Q^f h = \Lambda^{-1}h^{-1} \Lambda h = \ind$.
Our goal is therefore to show that the Markov operator~$Q_h$ fits the framework reminded in
Appendix~\ref{sec:tools}, in particular that it satisfies Assumptions~\ref{as:lyapunov}
and~\ref{as:minorization}.

Let us show that this operator satisfies Assumption~\ref{as:lyapunov} in Appendix~\ref{sec:tools}
with Lyapunov function~$Wh^{-1}$. We first note that the normalization $\| h \|_{\Linfty_W} =1$
implies that $Wh^{-1}\geq 1$. Using Assumption~\ref{as:lyapunovgeneral}, we obtain
\[
Q_h(Wh^{-1}) = \Lambda^{-1} h^{-1} Q^f W \leq \Lambda^{-1} h^{-1}\left(
\gamma_n W + b_n \ind_{K_n} \right)
\leq \frac{\gamma_n}{\Lambda}Wh^{-1} + \frac{b_n}{\Lambda  h}\ind_{K_n}.
\]
Noting that, for all $x\in K_n$,
\[
\Lambda h(x) = (Q^fh) (x) \geq \alpha_n \eta_n(h),
\]
with $\eta_n(h) >0$ from Lemma~\ref{lem:specQV}, the above inequality becomes
\begin{equation}
  \label{eq:lyapQh}
Q_h(Wh^{-1}) \leq \frac{\gamma_n}{\Lambda}Wh^{-1} + \frac{b_n}{\alpha_n \eta_n(h)}\ind_{K_n}. 
\end{equation}
Since~$\gamma_n$ can be taken arbitrarily small and $\eta_n(h)>0$ for any $n\geq 1$, we deduce
that~$Wh^{-1}$ is a Lyapunov function for~$Q_h$ in the sense of Assumption~\ref{as:lyapunov}
in Appendix~\ref{sec:tools}.
\begin{remark}
  \label{rem:Qh}
  Let us mention that, in order for~\eqref{eq:lyapQh} to define a Lyapunov condition in the
sense of Assumption~\ref{as:lyapunov}, it is not necessary to have $\gamma_n\to 0$ as
$n\to +\infty$. The existence of $n\geq 1$ such that $\gamma_n < \Lambda$ is  sufficient.
\end{remark}

We will now prove that: (i)~$Wh^{-1}$ has compact level sets, and (ii)~$Q_h$
satisfies Assumption~\ref{as:lyapunov} in Appendix~\ref{sec:tools} on any compact set~$K_n$,
that is~$\inf_{K_n} Q_h$ is lower bounded by some probability measure.
First, choosing $x_n\notin K_n$ in Assumption~\ref{as:lyapunovgeneral} leads to
\[
\Lambda h(x_n) = (Q^f h) (x_n) \leq \gamma_n W(x_n), 
\]
so that
\begin{equation}
  \label{eq:divWh}
\frac{W(x_n)}{h(x_n)} \geq \frac{\Lambda}{\gamma_n}.
\end{equation}
Since $\gamma_n \to 0$ as~$n\to +\infty$, the function~$Wh^{-1}$ diverges outside the compact
sets~$K_n$ defined in Assumption~\ref{as:lyapunovgeneral}. In other words,~$Wh^{-1}$ has
compact level sets, which shows (i).

Next, for $n\geq 1$, consider $\alpha_n >0$ and $\eta_n\in\PX$ as in Assumption~\ref{as:minogeneral},
so that, for any bounded measurable function $\varphi \geq 0$ and $x\in K_n$,
\[
Q_h\varphi(x) = \Lambda^{-1}\frac{Q^f( h\varphi)(x)}{h(x)}
\geq \frac{1}{\Lambda \sup_{K_n} h} \alpha_n
\eta_n(h \varphi) \geq \widetilde{\alpha}_n\, \widetilde{\eta}_{n}(\varphi),
\]
with 
\[
\widetilde{\alpha}_n = \alpha_n \frac{ \eta_n(h)}{\Lambda \sup_{K_n} h } >0,
\quad\ \widetilde{\eta}_n(\varphi)= \frac{\eta_n(h\varphi)}{\eta_n(h)}\in\PX.
\]
The latter expression is well-defined because, from Lemma~\ref{lem:specQV}, we know
that $0 < \eta_n(h)< +\infty $ for any $n\geq 1$. Moreover, $0 < \sup_{K_n} h < +\infty$
(since $h\in\Linfty_W(\X)$ and $\sup_{K_n} W <+\infty$ by
Assumption~\ref{as:lyapunovgeneral}), and this yields precisely (ii). Finally, (i) and (ii)
show that~$Q_h$ satisfies Assumption~\ref{as:minorization}, so that~$Q_h$ satisfies the assumptions of
Theorem~\ref{theo:HM}. As a result there exist a unique $\mu_h\in\PX$ and constants $c>0$,
$\bar{\alpha}\in (0,1)$ such that for any $\phi \in \Linfty_{W h^{-1}}(\X)$,
\[
\forall\, k\geq 0,\quad
\big\| Q_h^k\phi - \mu_h(\phi) \big\|_{\Linfty_{Wh^{-1}}} \leq c \bar{\alpha}^k
\| \phi - \mu_h(\phi) \|_{\Linfty_{Wh^{-1}}}.
\]
Moreover, the measure~$\mu_h$ satisfies $\mu_h(Wh^{-1})<+\infty$.

\section{Proof of Lemma~\ref{lem:estimateh}}
\label{sec:estimateh}

From~\cite[Proposition 1]{ferre2017error}, we obtain that $\Lc + f$ has a largest (in modulus)
eigenvalue~$\lambda$ with associated smooth eigenvector~$h$. Similarly, Lemma~\ref{lem:specQV}
shows that for any $\Dt \in(0,\Dt^*]$ the operator~$\Qdt^f$ has a largest (in modulus)
eigenvalue~$\Lambda_{\Dt}$ with continuous eigenvector~$h_{\Dt}$ (since~$\Qdt^f$ is assumed to
be strong Feller). Moreover, there is no restriction of
generality in normalizing~$h_{\Dt}$ so that $\eta(h_{\Dt})=1$.

We now turn to the estimate~\eqref{eq:lambdadt} on the spectral radius. In the notation
of~\cite{ferre2017error}, we have $\Lambda_{\Dt}=\e^{\Dt \lambdadt}$. A direct application
of~\cite[Theorem 3]{ferre2017error} then shows that there exist $\Dt^* > 0$ and $C >0$ such that
$\lambdadt = \lambda + \Dt c_{\Dt}$ with $| c_{\Dt} | \leq C $ for $\Dt \in (0,\Dt^*]$,
  which is the desired result.

Finally, since $\Qdt^f h_{\Dt}=\Lambda_{\Dt}h_{\Dt}$, the lower bound~\eqref{eq:minounif}
applied to $\varphi = h_{\Dt} \geq 0$ leads to
\[
\forall\, x \in \X, \quad
\left(\Qdt^f\right)^{\ceil{\frac{T}{\Dt}}}h_{\Dt}(x) = \Lambda_{\Dt}^{\ceil{\frac{T}{\Dt}}}
h_{\Dt}(x)  \geq \alpha \eta( h_{\Dt}).
\]
Using the estimate on~$\Lambda_{\Dt}$ and the normalization $\eta(h_{\Dt})=1$ we obtain,
for $\Dt\in (0,\Dt^*]$ (possibly upon decreasing~$\Dt^*$) and $ x \in\X$,
\[
h_{\Dt}(x) \geq  \Lambda_{\Dt}^{-\ceil{\frac{T}{\Dt}}} \alpha \eta( h_{\Dt})
\geq 
\alpha\, \e^{ -\Dt( \lambda +\Dt c_{\Dt}) \ceil{\frac{T}{\Dt}}  }
\geq \alpha\, \e^{-2 T |\lambda|}.
\]
A similar computation leads to an analogous upper bound, which shows~\eqref{eq:hdt}.

%------------------- BIBLIO ------------------------
\bibliographystyle{abbrv}
%\bibliography{bibliography}
%\bibliography{../../../../Bibliography/bibliography}

\end{document}